%% file: submit.tex
\documentclass[10pt, addressatend, dvipdfmx]{article}
\usepackage{mathtools,mathrsfs,amsthm,amsmath,amsfonts,amssymb}
\usepackage[numbers]{natbib}
\usepackage[colorlinks,citecolor=blue,urlcolor=blue]{hyperref}
\usepackage{graphicx}
\usepackage{verbatim}
\usepackage{epsfig}
\usepackage{color}
\usepackage{lipsum}
\usepackage{tikz}
\usetikzlibrary{intersections, calc, positioning, arrows.meta}
\usetikzlibrary{shapes.geometric}
\usetikzlibrary{shapes.misc}
\usetikzlibrary{patterns}
\tikzset{%
  >={Latex[width=2mm,length=2mm]},
            base/.style = {rectangle, rounded corners, draw=black,
                           minimum width=3cm, minimum height=1cm,
                           text centered, font=\sffamily},
  activityStarts/.style = {base },
       startstop/.style = {base },
    activityRuns/.style = {base},
         process/.style = {base, minimum width=2cm,                            font=\ttfamily},
}
\usepackage{caption,  xifthen,subcaption, enumerate,accents}

\definecolor{Blue}{rgb}{0,0,1}

\definecolor{Red}{rgb}{1,0,0}
\def\red{\color{Red}}

\definecolor{Green}{rgb}{0,1,0}

\definecolor{Yellow}{rgb}{1,1,0}

\DeclareFontFamily{OT1}{eusb}{} \DeclareFontShape{OT1}{eusb}{m}{n} {<5> <6> <7> <8> <9> <10> <11> <12> <14.4> eusb10}{}
\DeclareMathAlphabet{\eusb}{OT1}{eusb}{m}{n}

\DeclareFontFamily{OT1}{eusm}{} \DeclareFontShape{OT1}{eusm}{m}{n} {<5> <6> <7> <8> <9> <10> <11> <12> <14.4> eusm10}{}
\DeclareMathAlphabet{\eusm}{OT1}{eusm}{m}{n}

\DeclareFontFamily{OT1}{eufm}{} \DeclareFontShape{OT1}{eufm}{m}{n} {<5> <6> <7> <8> <9> <10> <11> <12> <14.4> eufm10}{}
\DeclareMathAlphabet{\mathfrak}{OT1}{eufm}{m}{n}

\DeclareFontFamily{OT1}{fraktura}{}
\DeclareFontShape{OT1}{fraktura}{m}{n} {<5> <6> <7> <8> <9> <10> <11> <12> <13> <14.4> [1.1] eufm10}{}
\DeclareMathAlphabet{\fraktura}{OT1}{fraktura}{m}{n}

\DeclareFontFamily{OT1}{cmfi}{} \DeclareFontShape{OT1}{cmfi}{m}{n} {<5> <6> <7> <8> <9> <10> <11> <12> <13> <14.4> [0.9] cmfi10}{}
\DeclareMathAlphabet{\cmfi}{OT1}{cmfi}{b}{n}

\DeclareFontFamily{OT1}{cmss}{} \DeclareFontShape{OT1}{cmss}{m}{n} {<5> <6> <7> <8> <9> <10> <11> <12> <13> <14.4> cmss10}{}
\DeclareMathAlphabet{\cmss}{OT1}{cmss}{m}{n}

\newtheoremstyle{thm}{1.5ex}{1.5ex}{\itshape\rmfamily}{} {\bfseries\rmfamily}{}{2ex}{}

\newtheoremstyle{def}{1.5ex}{1.5ex}{\slshape\rmfamily}{} {\bfseries\rmfamily}{}{2ex}{}

\newtheoremstyle{rem}{1.3ex}{1.3ex}{\rmfamily}{} {\itshape}
{} {1.5ex}{}

\theoremstyle{thm}
\newtheorem{theorem}{Theorem}[section]
\newtheorem{lemma}[theorem]{Lemma}
\newtheorem{claim}[theorem]{Claim}

\newtheorem{proposition}[theorem]{Proposition}

\newtheorem*{Main Theorem}{Main Theorem.}
\newtheorem{corollary}[theorem]{Corollary}
\newtheorem*{special theorem}{Lindeberg-Feller Theorem for Martingales}

\theoremstyle{def}

\theoremstyle{rem}

\numberwithin{equation}{section}

\renewcommand{\section}{\secdef\sct\sect}
\newcommand{\sct}[2][default]{%
\refstepcounter{section}
\vspace{0.7cm}
\centerline{\scshape\thesection.\ #1} \nopagebreak \vspace{0.2cm}}
\newcommand{\sect}[1]{%
\vspace{0.4cm} \centerline{\large\scshape\rmfamily #1}
\vspace{0.2cm}}

\newcommand{\ignore}[1]{{}}
\newcommand{\theprobspace}{{\mathbb X}}
\newcommand\blfootnote[1]{%
  \begingroup
  \renewcommand\thefootnote{}\footnote{#1}%
  \addtocounter{footnote}{-1}%
  \endgroup
}

\renewcommand{\subsection}{\secdef\subsct\sbsect}
\newcommand{\subsct}[2][default]{\refstepcounter{subsection}
\nopagebreak\vspace{0.45\baselineskip} {\flushleft\bf
\thesubsection~\bf #1.~}
\\*[3mm]\noindent
\nopagebreak}
\newcommand{\sbsect}[1]{\vspace{0.1cm}\noindent
\textbf{#1.~}\vspace{0.1cm}}

\renewcommand{\subsubsection}{%
\secdef \subsubsect\sbsbsect}
\newcommand{\subsubsect}[2][default]{%
\refstepcounter{subsubsection}
\nopagebreak
\vspace{0.15\baselineskip} \nopagebreak {\flushleft\rmfamily
\itshape\thesubsubsection
\ \rmfamily #1\/.}\ }
\newcommand{\sbsbsect}[1]{\vspace{0.1cm}\noindent
\rmfamily \itshape
\arabic{section}.\arabic{subsection}.\arabic{subsubsection} \
\sffamily #1\/.\ }

\newcommand{\N}{\mathbb N}

\newcommand{\BbbP}{\mathbb P}

\newcommand{\Z}{\mathbb Z}

\definecolor{Blue}{rgb}{0,0,1}

\def\myffrac#1#2 in #3{\raise 2.6pt\hbox{$#3 #1$}\mkern-1.5mu\raise 0.8pt\hbox{$#3/$}\mkern-1.1mu\lower 1.5pt\hbox{$#3 #2$}}

\newcommand{\quenchedP}{P}
\newcommand{\annealedP}{\BbbP}
\newcommand{\quenchedE}{E}
\newcommand{\annealedE}{{\mathbb E}}

\input{local_macros}

\begin{document}

\tikzstyle{decision} = [diamond, draw, fill=blue!20, 
    text width=4.5em, text badly centered, node distance=3cm, inner sep=0pt]
\tikzstyle{block} = [rectangle, draw, fill=blue!20, 
    text width=5em, text centered, rounded corners, minimum height=4em]
\tikzstyle{line} = [draw, -latex']
\tikzstyle{cloud} = [draw, ellipse,fill=red!20, node distance=3cm,
    minimum height=2em]

%
%
%

\title{Scaling limits for random walks on long range percolation clusters}

\author{Noam Berger \thanks{Technical University of Munich, Fakult\"at f\"ur Mathematik,
85748 Garching bei M\"unchen, 
Germany. {\tt noam.berger@tum.de}}
 \and Yuki Tokushige\thanks{Technical University of Braunschweig, Institut f\"{u}r Mathematische Stochastik,
38106 Universit\"atsplatz 2 Braunschweig, Germany
 {\tt yuki.tokushige@tu-braunschweig.de}}}



\maketitle

\begin{abstract}
We study limit laws for simple random walks on supercritical long-range
percolation clusters on $\Z^d, d\geq 2$. For the long range percolation model, the
probability that two vertices $x,y$ are connected behaves asymptotically as
$|x-y|^{-s}$ with $s \in [d + 1,d + 2)$. We prove that the scaling limit of simple
random walk on the infinite component converges to an $\alpha$-stable L\'evy process with $\alpha = s - d$. This complements the work of Crawford and Sly \cite{crawford2013simple}, who proved the corresponding result for the case $s \in (d, d+1)$.  The convergence holds in both the quenched and annealed senses.
\end{abstract}



\section{Introduction}

\label{sec:intro}\noindent

\subsection{background}
Long range percolation is a percolation model on $\Z^d$ where there can be an open edge between any pair of two distinct vertices. This stands in contrast to the more classical nearest neighbour percolation, where an edge can only span two `nearest-neighbor' vertices, i.e. two vertices of Euclidean distance $1$. For long range percolation, the probability of an edge being open is taken to be a function of the distance between its end vertices. This model has been studied extensively since its introduction by Schulman in 1983 \cite{Schul83}.

It was already noticed in Schulman's seminal paper that the most interesting and relevant regimes are when the connection probabilities decay like a power of the distance, i.e. the probabiltiy that the edge between $x$ and $y$ is open decays like $|x - y|^{-s}$, where $|\cdot|$ denotes the $\ell^2$ norm.

Combining the results of \cite{Schul83}, \cite{newman1986one} and \cite{aizenman1986discontinuity} we get that for $s \leq d$, an infinite component exists almost surely. In particular, in that case the degree of every vertex is infinite, and the entire lattice spans one connected component. For $d \geq 2$ and $s > d$ a phase transition for the existence of an infinite component exists. For $d=1$, when $1 < s \leq 2$ a phase transition exists, and for $s > 2$ an infinite component can exists only if the occupation probability of nearest-neighbour edges is 1. The case $d = 1, s = 2$ exhibits more complex behaviour, which we will not discuss in this paper.
See \cite{aizenman1986discontinuity,duminil2020long} for more detailed explanations.

Some of the main topics that have been studied in this context are critical behaviour \cite{aizenman1986discontinuity, baumler2022isoperimetric, hutchcroft2022sharp, hutchcroft2021critical, hutchcroft2021power, heydenreich2008mean, duminil2020long, chen2015critical, berger2002transience}, geometric features \cite{benjamini2008long, baumler2023distances, baumler2022behavior, benjamini2001diameter, berger2004lower, biskup2004scaling, biskup2011graph, biskup2019sharp, coppersmith2002diameter, ding2013distances, hao2023graph, wu2022sharp} and behaviour of the random walk on it \cite{BCKW21, berger2002transience, crawford2013simple, crawford2012simple, can2021spectral}.

The latter topic started with classification to recurrent and transient long range percolation \cite{berger2002transience}, and then evolved into the more advanced (and difficult) questions of heat kernels and scaling limits \cite{BCKW21, crawford2013simple, crawford2012simple, can2021spectral}. Most relevantly for the present paper, the main (but not only) result in \cite{crawford2013simple} is that when $d<s<d+1$, then the random walk scales to an $\alpha$-stable process, for $\alpha = s - d$. We prove in this paper that this continues to be true for $d+1 \leq s < \min(d+2, 2d)$. These results are in contrast to results obtained in the works
\cite{Barlow, berger-biskup2007,mathieu-piatniski2007}, where random walks on nearest-neighbour percolation clusters are shown to exhibit Gaussian fluctuation.

\subsection{notations and definitions}

  For $x\in \Z^d$, we write $|x|$ (resp. $|x|_p$) for the $\ell^2$ (resp. $\ell^p$) norm of $x$.
  For $x\in\Z^d$ and $r>0$, define an open ball $B(x,r):=\{y\in\Z^d:\ |x-y|<r\}$.

Fix $d\geq 1$ and $s > d$, and let $P(x), x\in\Z^d\setminus\{0\}$ be such that
  $P(x) = P(-x)$ for all $x$, and
  there exists $C$ such that $P(x) \sim C|x|^{-s}$, namely 
\begin{equation}\label{eq:deflrp}
\lim_{|j|\to\infty}\frac{P(j)}{C|j|^{-s}} = 1
\end{equation}

In this paper we mostly focus on the case $d\geq 2$ and $s \in [d + 1,d + 2)$.

We now consider a percolation model on $\Z^d$: an edge between $x$ and $y$ is {\em open} with probability $P(x-y)$, and the events of openness of edges are all independent of each other. A percolation configuration will usually be donated by $\omega$. 
In this article we assume that the model is {\it percolating}, namely, there exists an infinite connected component in $\omega$.
Moreover, in our setting it is shown in \cite{AKN87} that
\[
P\left[\mbox{there exists an unique infinite component}\ \mathcal{C}^{\infty}\right] = 1.
\]
 Therefore, we can assume the uniqueness of the infinite cluster in what follows.
Following \cite{crawford2013simple} we use $\mu$ to denote the distribution of the long range percolation, and $\mu_0$ to denote the distribution of the long range percolation conditioned on the (positive probability) event that $0$ is in $\mathcal{C}^{\infty}$.
We also use $\nu$ and $\nu_0$ for the measures $\mu$ and $\mu_0$ weighted by the degree of the origin, i.e.
\[
\nu(A) = \frac 1Z\int_A \deg_\omega(0)d\mu(\omega)
\]
and the same for $\nu_0$ and $\mu_0$, 
where the normalization constant $Z$ is such that $\nu$ is a probability measure.

As already explained in \cite{crawford2013simple} the advantage of $\nu$ and $\nu_0$ is that they are reversible measures for the random walk in the point of view of the particle, while being equivalent to respectively $\mu$ and $\mu_0$. $\nu_0$ is invariant and ergodic for the random walk in the point of view of the particle.

Conditioned on the event $\{0\in\mathcal{C}^{\infty}\}$, we consider the simple random walk $(X_n)$ on the configuration $\omega$, namely the Markov chain such that 
$X_0 = 0$ and conditioned on $X_0,\ldots,X_n$, $X_{n+1}$ is uniformly distributed among all $\omega$-neighbours of $X_n$, where we say that $y$ is an {\em $\omega$-neighbour} of $x$ if the edge $(x,y)$ is open in $\omega$.

The purpose of this paper is to prove that after proper scaling the random walk $(X_n)$ converges to an isotropic $\alpha$-stable process, with $\alpha = s - d$. This statement was proved in \cite{crawford2013simple} for $s\in (d,d+1)$ and conjectured for $s\in[d+1, d+2)$. In order to be able to properly state our main theorem, we first need to discuss the distribution of the random walk and the topology in which the convergence takes place.

\subsection{Quenched and annealed distributions for the simple random walk}
We consider two different possible measures for the random walk $(X_n)$ - the quenched and the annealed. First, for any $\omega \in \Omega$ we define $\quenchedP^\omega_x$ to be the law of the simple random walk as described above. If we write $\quenchedP^\omega$, we mean $\quenchedP^\omega_0$, namely the quenched law of the walk starting at the origin.

We also define the {\em annealed measure} on $(X_n)$ to be the average of the quenched based on one of the measures $\mu$, $\mu_0$, $\nu$ or $\nu_0$, namely for a measure $Q$ on $\Omega$ (we think of $Q\in\{\mu, \mu_0, \nu, \nu_0\}$) we define
\[
\annealedP_{Q}^x = \int_{\Omega}P^\omega_x dQ(\omega),
\]
where, as before, if we omit the superscript $x$ then we mean that the random walk started at the origin.

\subsection{The topology in which the convergence takes place}
Crawford and Sly noticed in \cite{crawford2013simple} that convergence cannot take place in the usual Skorohod topology. The reason is that it happens quite often that a long edge is traversed an even number of times in a short period of time. Such an occurrence scales in the Skorohod topology to a removable one-point discontinuity, which does not exists in the limiting $\alpha$-stable process.  \cite{crawford2013simple} therefore proved the convergence in the $L^q$ norm, $q\in[1,\infty]$. In this norm, those jumps even out and disappear. Nevertheless, one has to be clear regarding the exact choice of space in which the convergence takes place: The sample of the stable process is almost surely in the space $L^q[0,1]$. The distribution is however not necessarily in the space $L^q([0,1]\times \theprobspace)$, where $\theprobspace$ is the probability space. We therefore view the stable process, as well as the random walk, as $L^q[0,1]$-valued random variables on the probability space $\theprobspace$, and prove convergence in distribution in this setting.

\subsection{Statement of the main result.}
As we explained in Introduction, the purpose of this paper is to study the fluctuation of a SRW on the LRP cluster. In this direction, Crawford and Sly \cite{crawford2013simple} proved the following quenched stable limit theorem for scaling limits:
\begin{theorem}\label{thm:cs} \cite[Theorem 1.1]{crawford2013simple}
Let $d\geq1$ and $s\in(d,d+1)$. Let $(X_n)$ be a simple random walk on $\omega$ starting from the origin $0$ and for $n\in\mathbb{N}$ and $t\in[0,1]$ we define the rescaled process
$$X^n(t):=n^{-1/\alpha}X_{\lfloor nt\rfloor}.$$
Then there exists a measurable set $\Omega_0\subset \{0\in\mathcal{C}^{\infty}\}$ with $\mu_0(\Omega_0)=1$ such that for $\mu_0$-a.a. environment $\omega\in\Omega_0$ and $q\in[1,\infty)$, the distribution of $(X^n(t))_{t\in[0,1]}$ weakly converges on $L^q[0,1]$ to the distribution of a $d$-dimensional isotropic $\alpha$-stable process.
\end{theorem}
Note that Theorem \ref{thm:cs} does not cover the whole regime of stable fluctuations. Namely, they do not obtain the quenched stable limit theorem for $s\in[d+1,d+2)$. 
Our main result fills this gap.
\begin{theorem}\label{thm:main}
Let $d\geq2$ and $s\in[d+1,d+2)$. Let $(X_n)$ be a simple random walk on $\omega$ starting from the origin $0$ and for $n\in\mathbb{N}$ and $t\in[0,1]$ we define the rescaled process
$$X^n(t):=n^{-1/\alpha}X_{\lfloor nt\rfloor}.$$
Then there exists a measurable set $\Omega_0\subset \{0\in\mathcal{C}^{\infty}\}$ with $\mu_0(\Omega_0)=1$ such that for $\mu_0$-a.a. environment $\omega\in\Omega_0$ and $q\in[1,\infty)$, the distribution of $(X^n(t))_{t\in[0,1]}$ weakly converges on $L^q[0,1]$ to the distribution of a $d$-dimensional isotropic $\alpha$-stable process.
\end{theorem}
We give several remarks regarding Theorem \ref{thm:main}.
To begin with, we explain one of major technical issues to extend the quenched stable limit theorem to $d\geq2, s\in[d+1,d+2)$. In this case, we encounter a problem to do with the control of the sum of short jumps of the walk. A well-known fact related to this matter is that sample paths of $\alpha$-stable processes on $\mathbb{R}^d$ have bounded variation only if $\alpha\in(0,1)$. In Section 2, we will address this technical problem.\par
Our next remark is that Theorem \ref{thm:main} does not cover the case $d=1$. This is not an artifact of our proof but an essential issue. As a matter of fact, Crawford and Sly \cite{crawford2013simple} proved the following invariance principle.
\begin{theorem}\label{thm:s>2}\cite[Theorem 1.2]{crawford2013simple}
Suppose that $d=1$ and $s>2$. Assume \eqref{eq:deflrp} and $P(1)=1$.
Let $(X_n)$ be a simple random walk on $\omega$ starting from the origin $0$ and for $n\in\mathbb{N}$ and $t\in[0,1]$ we define the rescaled process
$$X^n(t):=n^{-1/2}\left(X_{\lfloor nt\rfloor}+(tn-\lfloor tn\rfloor)(X_{\lfloor nt\rfloor +1}-X_{\lfloor nt\rfloor})\right).$$
Then there exists a measurable set $\Omega_0\subset \Omega$ with $\mu(\Omega_0)=1$ such that for $\mu$-a.a. environment $\omega\in\Omega_0$, the distribution of $(X^n(t))_{t\in[0,1]}$ weakly converges on $C[0,1]$ to the distribution of a one-dimensional Brownian motion.
\end{theorem}
One implication of this theorem is that when $d=1$ there is no stable fluctuations for $s\in[d+1,d+2)$.
This is due to the presence of {\it cut points} for $d=1,s>2$, which is a unique feature of one-dimensional LRP models. See \cite{Schul83} for details. These cut points are also the reason why we assume $P(1)=1$ in Theorem \ref{thm:s>2}; the model does not percolate otherwise. For $d=1$, it seems plausible to believe that fluctuations in the case $s=2$ somehow interpolates between a $1$-stable process and a Brownian motion depending on the value of $C$ in \ref{eq:deflrp}. We recommend interested readers to consult \cite[Section 2.4]{BCKW21}.

\subsection{Outline of the proof and structure of the paper}
This paper is organized as follows: in Section 2, we will show that the sum of short jumps weakly converges to 0 after scaling. In Section 3, we will collect several technical estimates which will play an important role in the paper.
In Section 4, we will review the coupling procedure introduced in \cite{crawford2013simple}.
In Section 5, we will discuss how to approximate a SRW on the LRP cluster by the sum of {\it i.i.d.} random variables and prove Theorem \ref{thm:main}. Finally, in Section 6 we will show Proposition \label{prop:phi}, which is a key technical input for the proof of Theorem \ref{thm:main}.

\section*{Acknowledgement}
The second author was supported by the
EPSRC Grant 'Homogenization of
random walks: degenerate environments and long-range jumps'.

\section{Control of the short steps}\label{sec:shortsteps}
\newcommand{\katan}{{\xi}}
\newcommand{\hezka}{{\epsilon}}

The main (but not the only) technical difference between the current paper and \cite{crawford2013simple} is in the control over the short steps. In \cite{crawford2013simple} the steps of the walk up to time $n$ were divided into {\em short steps}, namely steps of length up to length $n^{1/\alpha - \hezka}$, and 
{\em long steps}, i.e. steps longer than $n^{1/\alpha - \hezka}$. Then one proves that the shorts steps vanish in the limit whereas the long steps converge to the stable process. In \cite{crawford2013simple} the case $\alpha \in(0,1)$ was considered, in which one can simply sum up the absolute values of the short steps and get a value which is small enough. In the present paper this cannot work; the sum of the absolute values of the short steps is  too large. We therefore need to prove that the short steps cancel each other in a satisfactory way. To do this we adapt an old argument due to Kesten (see, e.g.  \cite{kesten1986subdiffusive}).

The main statement in this section is the following proposition.

\begin{proposition}\label{prop:shortsteps}
Fix $\varepsilon > 0$. For almost every configuration $\omega \in \Omega$, we have that
\begin{equation}\label{eq:shortsteps}
W_k:= \max_{m<2^k} \left| \frac{1}{2^{k/\alpha}} \sum_{j=1}^m (X_j - X_{j-1})\cdot{\bf 1}_{\left\{|X_j - X_{j-1}| \leq 2^{(1/\alpha - \hezka)k}\right\}} \right|
\end{equation}
converges to $0$ in distribution under the measure $\quenchedP^{\omega}$ as $k$ goes to infinity, where the phrase ``for almost every $\omega$'' refers to any of the measures
$\mu$, $\mu_0$, $\nu$ and $\nu_0$.\par
Set $\rho=\rho(\varepsilon):=(1-\alpha/2)\varepsilon\ $. Then we have that
\begin{align}\label{ineq:short}
\annealedE_{\nu}\left[ W_k\right]=O(2^{-\rho k}).
\end{align}
Furthermore, we have the following estimate too:
\begin{align}\label{ineq:short-condition}
\annealedE_{\mu}\left[ W_k\  |\ \{\exists  x\in\mathbb{Z}^d\ \text{with}\ |x|>2^{(1/\alpha-\varepsilon)k}\ \text{s.t.}\ \omega_{0,x}=1\}\right]=O(2^{-\rho k}).
\end{align}
\end{proposition}

\begin{proof}
Let $n\in \N$ and  $\katan >0$. For every $x\in\Z^d$ we define its truncated local drift 
\[
L_\omega(x) = \sum_{y: |y-x| \leq \katan n^{1/\alpha}} y\cdot P_{\omega}(x,x+y),
\]
where $P_{\omega}(x,x+y) := \quenchedP^{\omega}_{x}(X_1 = x+y)$ is the transition probability under $\quenchedP^{\omega}$ from $x$ to $x + y$.

Then notice that 
\[
\left(M_l =  \sum_{k=1}^l \left\{(X_k - X_{k-1})\cdot{\bf 1}_{\left\{|X_k - X_{k-1}| \leq \katan n^{1/\alpha}\right\}} - L_\omega(X_{k-1})\right\}\right)_{l=0,\ldots,n}
\]
is a $P^{\omega}_x$-martingale.
We let $n=2^k$ and $\xi=n^{-\varepsilon}=2^{-\varepsilon k}$ later.
 The (annealed) variance of each step of $(M_l)$ is at most
\[
\sum_{j:|j| \leq \katan n^{1/\alpha}} |j|^2 |j|^{-s} \leq C\sum_{k = 1}^{\katan n^{1/\alpha}} k^{d+1-s} = C\sum_{k = 1}^{\katan n^{1/\alpha}} k^{1-\alpha} \leq
C \left(\katan n^{1/\alpha}\right)^{2-\alpha} = C\katan^{2-\alpha} n^{\frac{2}{\alpha} - 1}.
\]
Now write $Y_k = X_{n} - X_{n-k}, \ k=0,\ldots,n$. By stationarity and reversibility, $(Y_n)$ is also $P_\nu$ distributed, and 

\[
\sum_{k=1}^n (X_k - X_{k-1})\cdot{\bf 1}_{\left\{|X_k - X_{k-1}|\leq\xi n^{1/\alpha}\right\}} = -\sum_{k=1}^n (Y_k - Y_{k-1})\cdot{\bf 1}_{\left\{|Y_k - Y_{k-1}|\leq\xi n^{1/\alpha}\right\}}.
\]
Note that 
\[
\left(M'_l =  \sum_{k=1}^l \{(Y_k - Y_{k-1})\cdot{\bf 1}_{\left\{|Y_k - Y_{k-1}| \leq \katan n^{1/\alpha}\right\}} - L_\omega(X_{n-k})\}\right)_{l=0,\ldots,n}
\]
is a martingale with the same distribution as $(M_l)$.
Write
\begin{align}\label{def:w(n)}
W(n) = \max_{m\leq n}\left|\sum_{k=1}^m (X_k - X_{k-1})\cdot{\bf 1}_{\left\{|X_k - X_{k-1}| \leq \katan n^{1/\alpha}\right\}}\right|.
\end{align}
Then for every $m\leq n$,
\begin{align*}
&\ \ \ 2\sum_{k=1}^m (X_k - X_{k-1})\cdot{\bf 1}_{\left\{|X_k - X_{k-1}| \leq \katan n^{1/\alpha}\right\}} \\
&= \sum_{k=1}^m (X_k - X_{k-1})\cdot{\bf 1}_{\left\{|X_k - X_{k-1}| \leq \katan n^{1/\alpha}\right\}}  
+\sum_{k=n - m}^n (Y_k - Y_{k-1})\cdot{\bf 1}_{\left\{|Y_k - Y_{k-1}|\leq \xi n^{1/\alpha}\right\}}
\end{align*}

and therefore

\begin{align}\label{eq:decompos_kest}
2W
 &\leq
\max_{m\leq n} \left| \sum_{k=1}^m (X_k - X_{k-1})\cdot{\bf 1}_{\left\{|X_k - X_{k-1}|\leq \xi n^{1/\alpha}\right\}} \right| \nonumber\\
&\hspace{25mm}+ \max_{m\leq n}\left| \sum_{k=m-n}^n (Y_k - Y_{k-1})\cdot{\bf 1}_{\left\{|Y_k - Y_{k-1}|\leq \xi n^{1/\alpha}\right\}}\right|\nonumber
\\
&\leq \max_{m\leq n}\big|M_m\big| +  \max_{m\leq n}\big|M'_m\big| + \big|M'_n\big| \nonumber\\
&\hspace{10mm}+\sum_{k=1}^n \big|L_\omega(X_{k})\big|\cdot {\bf 1}_{\left\{|X_k - X_{k-1}|> \xi n^{1/\alpha}\right\}} +\sum_{k=0}^{n-1} \big|L_\omega(X_{k})\big|\cdot{\bf 1}_{\left\{|X_{k+1} - X_{k}|> \xi n^{1/\alpha}\right\}}.
\end{align}

Write
\[
V_1^{(n)} = \max_{m\leq n}\big|M_m\big| +  \max_{m\leq n}\big|M'_m\big| + \big|M'_n\big|,
\]
and 
\begin{align}\label{def:v_2}
V_2^{(n)} = \sum_{k=1}^n \big|L_\omega(X_{k})\big|\cdot {\bf 1}_{\left\{|X_k - X_{k-1}|> \xi n^{1/\alpha}\right\}} +\sum_{k=0}^{n-1} \big|L_\omega(X_{k})\big|\cdot{\bf 1}_{\left\{|X_{k+1} - X_{k}|> \xi n^{1/\alpha}\right\}}
\end{align}
for the two terms of \eqref{eq:decompos_kest}.

Now
\[
\annealedE_{\nu}(M_n ^2) =\annealedE_{\nu}(M_n^{\prime 2}) \leq n \cdot C\katan^{2-\alpha} n^{\frac{2}{\alpha} - 1} = C\katan^{2-\alpha} n^{\frac{2}{\alpha}},
\]
and by Doob's inequality
\[
\annealedE_{\nu}\left(\max_{m\leq n} M_m ^2\right) \leq 4\annealedE_{\nu}(M_n ^2) = 4 C\katan^{2-\alpha} n^{\frac{2}{\alpha}},
\]
and similarly for $(M'_n)$, so
\begin{equation}\label{eq:cv1}
\annealedE_{\nu}\left[(V_1^n)^2\right] \leq C\katan^{2-\alpha} n^{\frac{2}{\alpha}}.
\end{equation}

We now control the first moment of $V_2^{(n)}$. 
Note that 
$${\bf 1}_{\left\{\left\|X_k - X_{k-1}\right\| > \katan n^{1/\alpha}\right\}} \leq 
{\bf 1}\left\{\exists y\in\mathbb{Z}^d\ \text{with}\ |y| > \katan n^{1/\alpha}\ \text{s.t.}\ \omega(X_k, X_k + y) = 1 \right\},$$
 and that under the product measure $\mu$, the variables 
 $${\bf 1}\left\{\exists y\in\mathbb{Z}^d\ \text{with}\ |y| > \katan n^{1/\alpha}\ \text{s.t.}\ \omega(X_k, X_k + y) = 1 \right\}$$ 
 and $\big|L_\omega(X_{k})\big|$ are independent.

We get that
\begin{eqnarray}\label{ineq:loc-dri}
&& \annealedE_{\nu} 
\left[
\big|L_\omega(X_{k})\big| \cdot{\bf 1}_{\left\{|X_k - X_{k-1}| > \katan n^{1/\alpha}\right\}}
\right]\nonumber\\
&\leq &
\annealedE_{\nu} \left[
\big|L_\omega(X_{k})\big| \cdot{\bf 1}_{\{\exists y\in\mathbb{Z}^d\ \text{with}\ |y| > \katan n^{1/\alpha}\ \text{s.t.}\ \omega(X_k, X_k + y) = 1\}}\right]\nonumber\\
&= &
\annealedE_{\nu} \left[
\big|L_\omega(0)\big| \cdot{\bf 1}_{\{\exists y\in\mathbb{Z}^d\ \text{with}\ |y| > \katan n^{1/\alpha}\ \text{s.t.}\ \omega(0, y) = 1\}}
\right]\nonumber\\
&=&\frac{1}{\annealedE_{\mu}[\deg_{\omega}(0)]}
\annealedE_{\mu} 
\left[
\left|\sum_{|y|\leq\xi n^{1/\alpha}}y\cdot\omega(0,y)\right| 
\cdot{\bf 1}_{\{\exists y\in\mathbb{Z}^d\ \text{with}\ |y| > \katan n^{1/\alpha}\ \text{s.t.}\ \omega(0, y) = 1\}}
\right]\nonumber\\
&=&
C\annealedE_{\mu}\left[\left|\sum_{|y|\leq\xi n^{1/\alpha}}y\cdot\omega(0,y)\right|\right]
\cdot\mu\left(\exists y\in\mathbb{Z}^d\ \text{with}\ |y| > \katan n^{1/\alpha}\ \text{s.t.}\ \omega(0, y) = 1\right)
\end{eqnarray}
where the first equality follows from the invariance of $\nu$ w.r.t. the point of view of the walker, and the last inequality follows from H\"older's inequality and the independence of 
$\left|\sum_{|y|\leq\xi n^{1/\alpha}}y\cdot\omega_{0,y}\right|$ and $\left\{\exists_{|y| > \katan n^{1/\alpha}} \omega(0,  y) = 1 \right\}$ under $\mu$.

Therefore by \eqref{def:v_2} and \eqref{ineq:loc-dri}, we have that

\begin{align}\label{eq:ctrlv2}
&\ \ \ \ \ \annealedE_{\nu}(V_2^{(n)})\nonumber\\
&\leq 2Cn \annealedE_{\mu}\left[\left|\sum_{|y|\leq\xi n^{1/\alpha}}y\cdot\omega(0,y)\right|\right]
\cdot\mu\left(\exists y\in\mathbb{Z}^d\ \text{with}\ |y| > \katan n^{1/\alpha}\ \text{s.t.}\ \omega(0, y) = 1\right)\nonumber\\
&\leq Cn\cdot\annealedE_{\mu}\left[\left|\sum_{|y|\leq\xi n^{1/\alpha}}y\cdot\omega(0,y)\right|^2\right]^{1/2}\cdot (\xi n^{1/\alpha})^{-\alpha}\nonumber\\
&\leq Cn\cdot (\katan^{2-\alpha} n^{\frac{2}{\alpha} - 1})^{1/2}\cdot \xi^{-\alpha}n^{-1}\nonumber\\
&\leq C\xi^{\frac{2-\alpha}{2}}n^{\frac{2-\alpha}{2\alpha}}.
\end{align}

Now let $n = 2^k$ and $\katan = n^{-\hezka} = 2^{ -\varepsilon k}$.
Then the expression in \eqref{eq:shortsteps} is bounded by
\begin{equation}\label{eq:wdcmps}
\frac{V_1^{(2^k)}}{2^{k/\alpha}} + \frac{V_2^{(2^k)}}{2^{k/\alpha}}.
\end{equation}
We control each of the two terms separately though similarly.

By \eqref{eq:cv1},
\begin{equation}\label{eq:smb1}
\sum_{k=1}^\infty \annealedE_{\nu}\left[\left(\frac{V_1^{(2^k)}}{2^{k/\alpha}}\right)^2\right] 
\leq \sum_{k=1}^\infty 2^{-2k/\alpha}\cdot C\katan^{2-\alpha} \cdot 2^{2k/\alpha}
= C\sum_{k=1}^\infty 2^{-\hezka(2-\alpha)k } 
< \infty,
\end{equation}

and by \eqref{eq:ctrlv2} 
\begin{align}\label{eq:smb2}
\sum_{k=1}^\infty \annealedE_{\nu}\left[\frac{V_2^{(2^k)}}{2^{k/\alpha}}\right] 
&\leq C\sum_{k=1}^\infty 2^{-k/\alpha}\cdot2^{-\frac{(2-\alpha)\varepsilon}{2} k}\cdot 2^{\frac{2-\alpha}{2\alpha}k}\nonumber\\
&\leq C\sum_{k=1}^{\infty}2^{-(\frac{1}{2}+(1-\frac{\alpha}{2})\varepsilon) k}<\infty.
\end{align}

As 
\[
\annealedE_{\nu}\left[\left(\frac{V_1^{(2^k)}}{2^{k/\alpha}}\right)^2\right] 
= \annealedE_{\nu}\left[\annealedE_{\nu}\left[\left.\left(\frac{V_1^{(2^k)}}{2^{k/\alpha}}\right)^2 \right| \omega\right] \right]
\]
and
\[
\annealedE_{\nu}\left[\frac{V_2^{(2^k)}}{2^{k/\alpha}}\right]
= \annealedE_{\nu}\left[\annealedE_{\nu}\left[\left.\frac{V_2^{(2^k)}}{2^{k/\alpha}} \right| \omega \right]\right],
\]
from \eqref{eq:smb1}, \eqref{eq:smb1} and Borel-Cantelli we get that 
\[
\lim_{k\to \infty }\annealedE_{\nu}\left[\left.\left(\frac{V_1^{(2^k)}}{2^{k/\alpha}}\right)^2 \right| \omega\right] =
\lim_{k\to \infty }\annealedE_{\nu}\left[\left.\frac{V_2^{(2^k)}}{2^{k/\alpha}} \right| \omega \right] = 0,
\]
for $\nu$-almost every $\omega$,
and from Jensen's inequality we get that 
\[
\lim_{k\to \infty }\annealedE_{\nu}\left[\left.\left(\frac{V_1^{(2^k)}}{2^{k/\alpha}}\right) \right| \omega\right] = 0
\]
for $\nu$-almost every $\omega$.

Write 
\begin{align}\label{def:w_k}
W_k = \max_{m<2^k} \left| \frac{1}{2^{k/\alpha}} \sum_{j=1}^m (X_j - X_{j-1})\cdot{\bf 1}_{\left\{|X_j - X_{j-1}| \leq 2^{k(1/\alpha - \hezka)}\right\}} \right|,
\end{align}
the quantity from \eqref{eq:shortsteps}.
Then, from \eqref{eq:wdcmps} we get that
\[
\lim_{k\to \infty} \quenchedE^\omega \left[ W_k
\right] = 0
\]
for $\nu$-almost every $\omega$, and thus, for $\nu$-almost every $\omega$ we have that  $W_k$ converges to zero  in probability, and thus also in distribution.
Since $\nu_0$, $\mu_0$ are both absolutely continuous w.r.t. $\nu$, this holds for them too. $\mu$ conditioned on the event $\{\deg 0 > 0\}$ is also absolutely continuous with respect to $\nu$, and on the event $\{\deg 0 = 0\}$ the random walk never leaves $0$, so \eqref{eq:shortsteps} is always $0$.
Therefore if we set $\rho:=(1-\alpha/2)\varepsilon$, then the estimate \eqref{ineq:short} is also already proven along the way.\par
We finally prove \eqref{ineq:short-condition}. 
Suppose that there is a vertex $y\in\mathbb{Z}^d$ with $|y|>2^{(1/\alpha-\varepsilon)k}$ such that $\omega_{0,y}=1$. If such an $y$ is not unique, choose the minimal one with respect to some deterministic order.
 Let $\omega^{(y)}$ be a new environment obtained by declaring the edge $(0,y)$ to be closed, then $\omega^{(y)}$ is distributed as $\mu$ conditioned on $$\{\text{for any}\ x\in\mathbb{Z}^d\ \text{with}\ |x|>2^{(1/\alpha-\varepsilon)k},\ \text{we have}\ \omega_{0,x}=0\}.$$
 Note that the probability of this event tends to $1$ as $k\to\infty$.
Moreover, it suffices to consider the case where there is the unique such vertex $y$ since the probability that there are multiple long edges attached to $0$ is significantly smaller.
Let $X^{(0)}$ and $X^{(y)}$ be two independent random walks on $\omega^{(y)}$ started at $0$ and $y$ respectively. Then   
$$ \annealedE_{\mu}\left[ W_k\  |\ \{\exists  x\in\mathbb{Z}^d\ \text{with}\ |x|>2^{(1/\alpha-\varepsilon)k}\ \text{s.t.}\ \omega_{0,x}=1\}\right]$$
 is smaller than the sum of expectations on the LHS of \eqref{ineq:short} applied for $X^{(0)}$ and $X^{(y)}$. 
 Note that the distribution of $\omega^{(y)}$ is $\mu$ conditioned on a typical event.
 Therefore for sufficiently large $k$, the expectations \eqref{ineq:short} applied for $X^{(0)}$ and $X^{(y)}$ are both bounded above by two times the expectation of $W_k$ under $\mathbb{E}^{\mu}$.
 Therefore, we get the desired estimate by using \eqref{ineq:short}.
\end{proof}


\section{Some quantitative transience estimates}\label{sec:trnsest}
\newcommand{\mekadem}{\kappa}
\newcommand{\trvedge}{A}

In this section, we will recall some estimates which quantify the transience of a simple random walk on the LRP cluster. 

\subsection{Bounds on the size of finite clusters}
In this subsection, we introduce an upper bound on the size of second largest cluster, which is finite. 
We will need it to show that the effect of conditioning ${0\in\mathcal{C}^{\infty}}$, which breaks down independence, 
when we study a scaling limit of a simple random walk. 
\begin{lemma}\cite[Lemma 9.2]{crawford2013simple}\label{lem:aaa}
Consider the box $[-N,N]^d$ of size length $2N$ and sample all internal edges $\{x,y\}\in [-N,N]^d\times [-N,N]^d.$
Let $n_1\geq n_2\geq ...\geq n_m$ be a sequence of cluster sizes inside the box $[-N,N]^d$. Namely, $n_l$ is the size of the $l$-th largest open cluster in $[-N,N]^d$. Let $M$ be the largest open cluster in $[-N,N]-d$.
Then there exist $c_1,c_2,C,\upsilon_1,\upsilon_2>0$ independent of $N$, such that
\begin{align*}
&\mathbb{P}_{\mu}(n_2>(\log N)^{\upsilon_1})<c_1e^{-c_2(\log N)^2}\ \ \text{and}\\
&\mathbb{P}_{\mu}(0\ \text{is connected to the complement of }[-N,N]^d\ |\ 0\notin M)<CN^{-\upsilon_2}
\end{align*}
for any $N>0$.
\end{lemma}

\subsection{Heat kernel estimates}
The paper \cite{crawford2012simple} and the later work \cite{can2021spectral} proved bounds on the heat kernel of the random walk on long range percolation. In \cite{crawford2013simple}, Crawford and Sly used these bounds to prove that various unwanted events are unlikely. The likely absence of those events was then used in the proof of the main result.
Here we provide an estimate in the same spirit of those in \cite{crawford2013simple}.
The following heat kernel estimates will play a very important role in what follows. 
\begin{proposition} \cite[Theorem 1]{crawford2012simple}\label{prop:hke}
Suppose either $s\in(d,d+2), d\geq2$, or $s\in(1,2), d=1$. Let $\alpha:=s-d$. Then there exist a measurable event $\Omega_1\subset\Omega$ with $\mu(\Omega_1)=1$, 
deterministic constants $C_1>0$ and $\theta>0$, depending on $(P(x))_{x\in\mathbb{Z}^d}$, and a family of random variables $T_x(\omega)>0$ with the following properties: 
\begin{enumerate}
\item for $x\in\mathbb{Z}^d$, we have $T_x(\omega)<\infty$ a.s. if and only if $x\in\mathcal{C}^{\infty}(\omega).$
\item For $x,y\in\mathcal{C}^{\infty}(\omega),$ it holds that
\begin{equation*}
P^{\omega}_x(X_n = y) \leq C_1 {\rm deg}^{\omega}(y)n^{-\frac{d}{\alpha}} \big(\log n\big)^{\theta}
\end{equation*}
for $n\geq T_x(\omega)\vee T_y(\omega).$
\item For any $\zeta>0$, there exists $C(\zeta)>0$ such that
$$\mu(T_x>k|x\in\mathcal{C}^{\infty}(\omega))\leq C(\zeta)k^{-\zeta}.$$
\end{enumerate}

\end{proposition}
Note that Can, Croydon and Kumagai \cite{can2021spectral} obtained matching lower bounds  under the additional assumption $P(y)=1$ for all $|y|_1=1$.
\par
In the proof, we will need the following corollary of Proposition \ref{prop:hke}
\begin{corollary}\label{cor:no-return}
Suppose that $d\geq2, s\in(d,d+2)$ and the system under consideration is percolating.
Then there exist $\varepsilon'\in\left(0,1-\frac{\alpha}{d}\right), \eta_3, \delta>0$
such that for any $i\geq0$ it holds that
\begin{align*}
\mathbb{P}_{\mu}\left(X_i\in \mathcal{C}^{\infty}\ \text{and}\ \exists i'\geq i+2^{(1-\varepsilon')k}\ \text{such that}\ X_{i'}\in B(X_i,2^{\delta k})\right)=O(2^{-\eta_3 k}).
\end{align*}
\end{corollary}
\begin{proof}
Noticing that we chose $\varepsilon'$ so that $\frac{d}{\alpha}\cdot(1-\varepsilon')>1$,
this estimate follows immediately from Proposition \ref{prop:hke}.
\end{proof}

\subsection{Transience estimates}
Following \cite{crawford2013simple} we consider many samples of the random walk in the same environment. For given $k$, we first sample the environment $\omega$ according to $\nu_0$, and then we sample $k^3$ random walks $\left(\big(X_i^\ell\big)_{i=0,\ldots 2^k}\right)_{\ell=1,\ldots k^3}$, all distributed according to $\quenchedP^{\omega}_0$ and independent of each other conditioned on $\omega$. By a slight abuse of notation we use $\annealedP_{\nu}$ to denote the joint distribution of those walks. A more precise, though more cumbersome, notation would be
\[
\int_\Omega \left(\quenchedP^\omega\right)^{k^3} d\nu(\omega).
\]
We will recall several estimates shown in \cite{crawford2013simple} which play an important role for the coupling argument to be introduced in the next section.\par
The first estimate asserts that for any $i,j\in[k^3]$ with $i\neq j$, the walks $(X^i_l)_{0\leq l\leq 2^k},(X^j_l)_{0\leq l\leq 2^k}$ do not find the same long edge with high probability. Define the event
\begin{align*}
F^*(k)&:=\{\exists l,l'\in[2^k], \exists u,v\in\mathbb{Z}^d\ \text{s.t.}\ \omega_{u,v}=1,\ 
|u-v|\geq 2^{(1/\alpha-\varepsilon)k}\\
&\hspace{85mm} \text{and}\ X_l^1, X_{l'}^2\in\{u\}\cup\{v\}\}.
\end{align*}
\begin{lemma}\label{lem:cs4.1} \cite[Proposition 4.1]{crawford2013simple}
There exist $\eta_1>0$ and $C>0$ such that
$$\mathbb{P}_{\mu}(F^*(k))\leq C2^{-\eta_1 k}.$$
\end{lemma}

The next estimate allows us to ignore the effect of various undesirable behaviors of random walks. We start with introducing a family of events which we do not wish to occur. We refer to \cite[Section 4, 9]{crawford2013simple} for detailed explanations.\par
Let $\gamma,\delta>0$ be positive numbers which we later choose to be sufficiently small. 
For $j\in[2^k]$ define
\begin{align*}
D_j(\varepsilon,k):=\{&\exists v\in\mathbb{Z}^d\ \text{with}\ \omega_{X_j^1,v}=1,|v-X_j^1|>2^{(1/\alpha-\varepsilon)k}
\text{and}\ \exists J\in[2^k]\\
&\hspace{60mm} \text{s.t.}\ X^1_J=v,\ (o,v)\notin\{(X^1_i,X^1_{i+1})\}_{i\leq J}\}.
\end{align*}
This is the event that there is a long edge at $X_j^1$ and the walk visits its other endpoint $v$ without crossing $(0,v)$.\par
We next let
\begin{align*}
E_j(\varepsilon,\delta,k):=\{&\exists u,v\in\mathbb{Z}^d\ \text{with}\ |v-X_j^1|>2^{(1/\alpha-\varepsilon)k}, \min(|u-X_j^1|,|u-v|)\geq2^{\delta k}\\
 &\hspace{45mm}\text{s.t.}\ \omega_{X_j^1,v}=1\ \text{and either}\ \omega_{X_j^1,u}=1\ \text{or}\ \omega_{u,v}=1\}.
\end{align*}
This is the event that there is a long edge at $X_j^1$ and one of its endpoint is attached to another edge of length greater than $2^{\delta k}$.\par
Set
\begin{align*}
F_j(\varepsilon,\gamma,k):=\{&\exists i\in[j+2^{\gamma k+1},2^k]\ \text{and}\ \exists v\in\mathbb{Z}^d\ 
\text{s.t.}\ \omega_{X_j^1,v}=1,\\
&\hspace{40mm} |v-X_j^1|\geq 2^{(1/\alpha-\varepsilon)k}\ \text{and}\ X_i^1\in\{X_j^1\}\cup\{v\}
\}.
\end{align*}
This is the event that there is a long edge at $X_j^1$ and the walk visits either $X_j^1$ or $v$ after
time $j+2^{\gamma k+1}$.\par
We finally define 
\begin{align*}
G_j&(\varepsilon,\gamma, \delta,k):=\{\exists v\in\mathbb{Z}^d\ \text{s.t.}\ \omega_{X_j^1,v}=1,\ |v-X_j^1|>2^{(1/\alpha-\varepsilon)k}\}\\
\bigcap&\big\{v\notin\{X^1_{j+1},X^1_{j+2},...,X^1_{j+2^{\gamma k+1}}\}\ \text{for}\ \forall v\in\mathbb{Z}^d\ \text{with}\ \omega_{X_j^1,v}=1, |v-X_j^1|>2^{(1/\alpha-\varepsilon)k}\big\}\\
\bigcap&\left\{\max_{0\leq l\leq 2^{\gamma k}}|X^1_{l+j}-X_j^1|>2^{\delta k}\right\}.
\end{align*}
This event guarantees that the walk does not exit a ball centered at $X_j^1$ with radius $2^{\delta k}$ without crossing a long edge.   
We want to exclude the occurrence of these events for any $j\in[2^k]$. To do so, we define 
\begin{align*}
\mathscr{D}(\varepsilon,k)&:=\bigcup_{j=0}^{2^k}D_j(\varepsilon,k),\ \ \ \ \ \ \ \ \ \ \ 
\mathscr{E}(\varepsilon,\delta,k):=\bigcup_{j=0}^{2^k}E_j(\varepsilon,\delta,k)\\
\mathscr{F}(\varepsilon,\gamma,k)&:=\bigcup_{j=0}^{2^k}F_j(\varepsilon,\gamma,k),\ \ \ \ 
\mathscr{G}(\varepsilon,\gamma,\delta,k):=\bigcup_{j=0}^{2^k}G_j(\varepsilon,\gamma,\delta,k).
\end{align*}
We finally set
\begin{align*}
\mathscr{H}(\varepsilon,\gamma,\delta,k):=\mathscr{D}(\varepsilon,k)\cup
\mathscr{E}(\varepsilon,\delta,k)\cup
\mathscr{F}(\varepsilon,\gamma,k)\cup
\mathscr{G}(\varepsilon,\gamma,\delta,k).
\end{align*}

\begin{lemma}\label{lem:cs4.2} \cite[Proposition 4.2]{crawford2013simple}
Let $d\geq1$ and $s\in(d,(d+2)\wedge 2d)$. For any $0<\delta<1$, there exist  constants $C>0$ and $\gamma,\eta_1>0$ sufficiently small such that
\begin{align*}
\mathbb{P}_{\mu}(\mathscr{H}(\varepsilon,\gamma,\delta,k))\leq C2^{-\eta_1 k}.
\end{align*}

\end{lemma}

\section{Coupling}\label{sec:coup}
This section is devoted to a review of the coupling argument in \cite[Section 5]{crawford2013simple}, which will be a fundamental ingredient of our proof for the quenched stable limit theorem. 
The coupling involves two phases: the {\it main phase} and the {\it special phase}. We are in the special phase when the walk under consideration recently found a new long edge, and we are in the main phase otherwise. Note that we are typically in the main phase, but the contribution to the displacement of the walk predominantly comes from the special phase, which is a feature of heavy-tailed phenomena.

\subsection{Independent variables}
We will introduce several sequences of {\it i.i.d.} random variables that will appear in the coupling argument. 
\par
For $i\geq1, l\in[k^3]$ and $x\in\mathbb{Z}^d$, let $w^l_i$ be independent Bernoulli random variables such that
$$\mathbb{P}_{\mu}(w_i^l(x)=1)=P(x).$$
These variables will be used in the coupling procedure to describe new long edges found by simple random walks.\par
We next define a family of independent geometric distributions. For $p\in[0,1]$, we denote by ${\tt Geom}(p)$ a geometric random variable with parameter $p$. Namely,
$$\mathbb{P}_{\mu}({\tt Geom}(p)=k)=(1-p)^kp.$$ 
Let ${\tt Uni}(0),{\tt Uni}(1),...$ be a sequence of  {\it i.i.d.} uniform random variables on $[0,1]$.
We then define a family of geometric random variables by
$$R(t)=\min\{i\geq0:{\tt Uni}(i)<t\},\ \ t\in[0,1].$$ 
Then it is obvious that $R(t)$ is decreasing in $t$ and distributed as the geometric distribution with parameter $t$. For $i\geq1, l\in[k^3]$, let $R_i^l(t)$ and $\tilde{R}_i^l(t)$ be independent copies of $R(t)$. These families will be used in the coupling to deal with a technical problem to do with the number that the walk crosses a long edge before escape. 
See Claim \ref{claim:cross} for details.\par
We will need another independent random variables which we use for describing local return probabilities and local degrees around endpoints of newly discovered long edge. For $v\in\mathbb{Z}^d$ and $k\in\mathbb{N}$, we define
$$\tilde{d}^{\omega}(v):=\#\{u\in\mathbb{Z}^d:\ |u-v|\leq 2^{\delta k}\ {\rm and}\ u\ {\rm is\ an}\ \omega\mathchar`-{\rm neighbor\ of}\ v\},$$
which we call the local degree.
Next, we let $\tilde{p}_v=\tilde{p}_v(k)$ be the quenched probability that a walk that started at $v$ and is conditioned to stay in a local neighborhood $\{u\in\mathbb{Z}^d:|u-v|\leq 2^{\delta k}\}$ returns to $v$ within time $2^{\gamma k}$. Also, we set $\tilde{p}_v$ to 1 if there is no open edges between $u$ and vertices $v$ with $|u-v|\leq 2^{\delta k}$.
We call $\tilde{p}_v$ the local return probability. 
For $v\in\mathbb{Z}^d$ and $k\in\mathbb{N}$, we define $\mathfrak{r}^l_i$ and $\mathfrak{d}^l_i$ as independent copies of $\tilde{p}_0$ and $\tilde{d}^{\omega}(0)$ under $\mathbb{P}_{\mu}$, respectively

\subsection{Summary of the coupling}
We will summarize the procedure to couple $k^3$ independent walks $\{\big(X_i^\ell\big)_{0\leq i\leq 2^k}\}_{1\leq l\leq k^3}$ and the environment $\omega$ with {\it i.i.d.} random variables introduced in the previous subsection. 
There are the following basic rules for the coupling: suppose that 
for all $1\leq m\leq l-1$, we have already run the $m$-th walk up to time $2^k$ 
and the $l$-th walk has made $i$ steps so far. Then
\begin{itemize}
\item if there is some $i'$ with $i'\in[i-2^{\gamma k+1}-1,i-1)$ such that at time $i'$ the $l$-th walk found a new long edge, we are in the special phase.  
\item If there is no such an $i'\in[i-2^{\gamma k+1}-1,i-1)$, we are in the main phase.
\end{itemize}
We set $X_0^l:=0$ for any $l\in[k^3]$. Therefore we are mainly interested in behaviors of the walks on the event $\{0\ \text{belong to the infinite cluster}\}$.
However we do not need to consider the conditioning by this event. As a matter of fact, Lemma \ref{lem:aaa} asserts that when $0$ is not in the infinite cluster, the size of connected component of $0$ is so small that the displacement of the walks will vanish after the scaling $2^{-k/\alpha}$.
It therefore does not affect the scaling limit.\par 
We begin with an explanation of the main phase since it is less involved.
\par\medskip
{\bf Main phase:} The construction of the main phase is divided into two cases, roughly speaking, depending on whether $X_i^l$ is already visited or not.\par\medskip
\underline{{\it Case 1 (when $X_i^l$ is already visited):}} we first precisely define the meaning of the assumption that $X_i^l$ is already visited. 
For $j\in[2^k], 1\leq l\leq [k^3]$, define 
\begin{align*}
\mathcal{W}_{j,l}:=\{x\in\mathbb{Z}^d:\ \exists n\in[j]\ \text{s.t.}\ X^l_n=x,\ \ \text{or}\ \ \exists m\leq l-1, \exists n'\in[2^k]\ \text{s.t.}\ X^m_{n'}=x\},
\end{align*}
which is the set of vertices we have found so far. Then we set
\begin{align*}
\mathcal{W}^+_{j,l}&:=\mathcal{W}_{j,l}
\cup\{y\in\mathbb{Z}^d:\ \exists z\in\mathcal{W}^+_{j,l}\ \text{with}\ |y-z|> 2^{(1/\alpha-\varepsilon)k}\ \text{s.t.}\ \omega_{y,z}=1\},
\end{align*}
which is the union of $\mathcal{W}_{j,l}$ and endpoints of long edges attached to a vertex in $\mathcal{W}_{j,l}$.
 The actual assumption for {\it Case 1} is that $X_i^l\in\mathcal{W}^+_{i-1,l}$.\par
  It is shown in \cite[Claim 5.1, page 460]{crawford2013simple} that
  there exists a constant $\eta_2>0$ such that
   the probability that there exist $i\in[2^k], l\in[k^3]$ and $y\in\mathbb{Z}^d$ with $|y-X_i^l|>2^{(1/\alpha-\varepsilon)k}$ such that $\omega_{y,X^l_i}=1$ is less than $2^{-\eta_2 k}$.
   This follows from Lemma \ref{lem:cs4.1} and Lemma \ref{lem:cs4.2}; more specifically, we use the upper bounds on the probabilities of $F^*(k)$, $\mathscr{D}(\varepsilon,k)$ and $\mathscr{F}(\varepsilon,\gamma,k)$.\par
 It implies that with high probability the walks do not find a long edge which is already discovered once again $2^{\gamma k+1}$ time after its first visit. If this unlikely event occurs, we call it a {\it type one error}.\par\medskip

\underline{{\it Case 2 (when $X_i^l$ is a new vertex):}} in this case we suppose that $X_i^l\notin\mathcal{W}^+_{i-1,l}$, which implies that we are at a new vertex. We first reveal all edges attached to $X_i^l$ and all vertices in the ball centered at $X_i^l$ of radius $2^{\delta k}$. In particular, 
we express long edges by independent random variables introduced in the previous subsection,
 namely we couple the environment and independent random variables so that $\omega_{x,X_i^l}=w^l_i(x-X_i^l)$ for each $x\in\mathbb{Z}^d$ with $|x-X_i^l|>2^{(1/\alpha-\varepsilon)k}$.
The next procedure is divided into several cases depending on the number of long edges attached to $X_i^l$.\par\medskip
(1) Suppose that $\sum_{x:|x-X_i^l|>2^{(1/\alpha-\varepsilon)k}}w^l_i(x-X_i^l)=0$, namely there is no long edge attached to $X_i^l$. Then we simply choose the next step $X_{i+1}^l$ uniformly among the neighborhood of $X_i^l$ and we are still in the main phase.
\par\medskip
(2) Suppose that $\sum_{x:|x-X_i^l|>2^{(1/\alpha-\varepsilon)k}}w^l_i(x-X_i^l)\geq2$ or $\sum_{x:|x-X_i^l|>2^{(1/\alpha-\varepsilon)k}}w^l_i(x-X_i^l)=1$ and there exist $x,y\in\mathbb{Z}^d$ such that $w^l_i(x-X_i^l)=1$, $|x-y|\leq2^{\delta k}$ and $y\in\mathcal{W}^+_{i-1,l}$.
This means that either there are more than two long edges attached to $X_i^l$, or there is exactly one long edge $(x,X_i^l)$ attached to $X_i^l$ and the other endpoint $x$ is within distance $2^{\delta k}$ from a vertex in $\mathcal{W}^+_{i-1,l}$. It is shown in \cite[Claim 5.2]{crawford2013simple} that 
the probability that this event occurs for some $i\in[2^k], l\in[k^3]$ is less than $2^{-\eta_2 k}$,
 and when it occurs we call it a {\it type two error}. This is a consequence of the bound on the probability of $\mathscr{E}(\varepsilon,\delta,k)$ and an elementary computation of the probability of the former case.\par
We then choose the next step $X_{i+1}^l$ uniformly among the neighborhood of $X_i^l$ and we are still in the main phase.\par\medskip
(3) The only scenario which is left to us is that there is exactly one $x\in\mathbb{Z}^d$ with $|x-X_i^l|>2^{(1/\alpha-\varepsilon)k}$ such that $\omega_{x,X_i^l}=1$, and random variables $\omega_{y,z}$ ($y,z\in B(x,2^{\delta k})$) are all unrevealed and therefore are independent of the procedure up to now. Now we reveal the edges $\omega_{y,z}$ ($y,z\in B(x,2^{\delta k})$) then we can couple the local return probability and the local degree with independent random variables as follows:
$$(\tilde{p}(x),\tilde{d}^{\omega}(x))=(\mathfrak{r}^l_i,\mathfrak{d}^l_i).$$
We also reveal all edges adjacent to $x$, namely $\omega_{x,y}$ for $y\in\mathbb{Z}^d$.
It is proved in \cite[Claim 5.3]{crawford2013simple} that the probability that there exist $i\in[2^k], l\in[k^3]$ such that we have either $d^{\omega}(X_i^l)\neq \tilde{d}^{\omega}(X_i^l)+1$ or $d^{\omega}(x)\neq \tilde{d}^{\omega}(x)+1$
is less than $2^{-\eta_2 k}$. 
If this event occurs, we call it a {\it type three error}.
This means that with high probability every edge other than $(x,X_i^l)$ which is attached to either $X_i^l$ or $x$ has length smaller than $2^{\delta k}$ whenever $|x-X_i^l|>2^{(1/\alpha-\varepsilon)k}$.
This is ensured by the bound on the probability of $\mathscr{E}(\varepsilon,\delta,k)$.
In this case we enter the special phase we describe below.\par\medskip

{\bf Special Phase:} We now explain the construction of the special phase which is significantly more involved than the main phase. 
Let us first give a couple of remarks about the construction of the special phase before going into details. In order to understand the scaling limit of the walk, when a new long edge is found what matters is which endpoint of it the walk eventually ends up with. 
Since a RW on the infinite cluster is transient, the walk only crosses a given edge finitely often. Therefore, what we have to know is the parity of the number of crossings of a new long edge. 
We will review below the procedure given in \cite{crawford2013simple} which determines this parity. 
A very important feature of the procedure is that it only involves information about the local structure of the configuration $\omega$ around two endpoints of a long edge, which allows us to retain independence  after the special phase. \par
Suppose that the $l$-th walk at time $i$ entered the special phase, which means, by definition, that $X_i^l$ is attached to a unique long edge $(x,X_i^l)$ with $|x-X_i^l|>2^{(1/\alpha-\varepsilon)k}$ which has not beed discovered before. 
Recall that we have $(\tilde{p}_{x},\tilde{d}^{\omega}(x))=(\mathfrak{r}^l_i,\mathfrak{d}^l_i)$ by the construction of the coupling. 
Define the localized graph $V^*$ consisting of vertices in $B(X_i^l,2^{\delta k})\cup B(x,2^{\delta k})$ and edges 
$$\{(x,X_i^l)\}\cup\{(y,z)\ \text{s.t.}\ y,z\in B(X_i^l,2^{\delta k}), \omega_{y,z}=1\}\cup\{(y',z')\ \text{s.t.}\ y',z'\in B(x,2^{\delta k}),\ \omega_{y',z'}=1\}.$$
Let $(Y_t)$ be a random walk on $V^*$ independent of the procedure up to this point. This walk will be coupled with the original walk $X^l$ later. From the time at which we enter the special phase, we have to wait for a while since the walk may cross a long edge back and forth. This waiting time needs to be sufficiently long so that the walk does not visit either of endpoints of a long edge after it with high probability.  
For this purpose, we define
\begin{align*}
\tau^*:=\inf\{t>2^{\gamma k}:\ Y_t\notin\{x\}\cup\{X_i^l\}\ \text{for}\ t-2^{\gamma k}\leq\forall t'\leq t\}.
\end{align*}
Namely, $\tau^*$ is the smallest time that we have $Y_t\notin\{X_i^l,x\}$ in the last $2^{\gamma k}$ steps, 
and what we want to know is whether $Y_{\tau^*}\in B(X_i^l,2^{\delta k})$ or $Y_{\tau^*}\in B(x,2^{\delta k})$.\par
The next step is to define the procedure determining the parity of the number of crossings of $(x,X_i^l)$. An important thing is, as mentioned before, that this procedure only involves the structure of $V^*$. 
To do so, we consider the following three types of excursions of $(Y_t)$ started at $X_i^l$.
\begin{enumerate} 
\item[(1)] Move from $X_i^l$ to $x$. This occurs with probability $\frac{1}{1+\tilde{d}^{\omega}(X_i^l)}$.
(Recall the definition of type three errors.)
\item[(2)] Move from $X_i^l$ to another vertex in $B(X_i^l,2^{\delta k})$, then behave as a random walk in $B(X_i^l,2^{\delta k})$ and return to $X_i^l$ within the next $2^{\gamma k}$ steps. This occurs
with probability $\frac{\tilde{p}(X_i^l)\tilde{d}^{\omega}(X_i^l)}{1+\tilde{d}^{\omega}(X_i^l)}$. 
\item[(3)] Move from $X_i^l$ to another vertex in $B(X_i^l,2^{\delta k})$, then behave as a random walk in $B(X_i^l,2^{\delta k})$ and not return to $X_i^l$ within the next $2^{\gamma k}$ steps. This occurs
with probability $\frac{(1-\tilde{p}(X_i^l))\tilde{d}^{\omega}(X_i^l)}{1+\tilde{d}^{\omega}(X_i^l)}$. 
\end{enumerate}
Define $R(X_i^l)$ as the number of excursions of type (1) performed by $(Y_t)$ started at $X_i^l$ before it makes an excursion of type (3). This variable measures the number of crossings of the long edge from $X_i^l$ to $x$ before the original walk escapes from $X_i^l$. We then consider the same construction for the walk started at $x$ and define $R(x)$ analogously. An important implication of the construction explained above is the following claim. 
\begin{claim} \label{claim:cross}\cite[Lemma 5.4]{crawford2013simple}\ \ The variables $R(X_i^l)$ and $R(x)$ are independent. Moreover they are distributed as
\begin{align*}
{\tt Geom}\left(
\frac{(1-\tilde{p}(X_i^l))\tilde{d}^{\omega}(X_i^l)}{1+(1-\tilde{p}(X_i^l))\tilde{d}^{\omega}(X_i^l)}
\right)\ \ \ \text{and}\ \ \ {\tt Geom}\left(\frac{(1-\tilde{p}(x))\tilde{d}^{\omega}(x)}{1+(1-\tilde{p}(x))\tilde{d}^{\omega}(x)}\right)
\end{align*}
respectively, and we have that
\begin{align*}
 Y_{\tau^*}\in B(X_i^l,2^{\delta k})\ \text{when}\ R(X_i^l)\leq R(x)\ \text{and} \ Y_{\tau^*}\in B(x,2^{\delta k})\ \text{when}\ R(X_i^l)> R(x).
\end{align*}
Note that the fact that the walk started at $X_i^l$ is the cause of the asymmetry between $R(X_i^l)\leq R(x)$ and $R(X_i^l)> R(x)$.
\end{claim}
 We want to couple the random walk $Y$ to the original walk $X^l$ in such a way that it holds that
 \begin{align*}
 R(X_i^l)=R_i^l\left(\frac{(1-\tilde{p}(X_i^l))\tilde{d}^{\omega}(X_i^l)}{1+(1-\tilde{p}(X_i^l))\tilde{d}^{\omega}(X_i^l)}\right)\ \ \text{and}
 \ \ R(x)=\left( \frac{(1-\tilde{p}(x))\tilde{d}^{\omega}(x)}{1+(1-\tilde{p}(x))\tilde{d}^{\omega}(x)} \right).
 \end{align*}
To do so, we describe how to run the walk $X^l_t$ step by step for $t\in[i,i+2^{\gamma k+1}-1]$ and exclude unwanted events.
\begin{itemize}
\item We want to exclude the event that the walk finds a new long edge during the special phase.
It is claimed in \cite[Claim 5.6]{crawford2013simple} that the probability that for some $l\in[k^3]$, the walk $X^l$ finds a new long edge during some special phase is less than $2^{-\eta_2 k}$, and if this occurs, we call it a {\it type four error}.
This bound on the probability can be shown by using the union bound and the following elementary observation: the number of long edges the walk find up to time $N$ is stochastically dominated by the sum of $N$ {\it i.i.d.} Bernoulli random variables with success probability $2^{-(1-\alpha\varepsilon)k}$.
\item For $t\in[i,i+2^{\gamma k+1}-1]$, choose $X^l_{t+1}$ uniformly among the neighborhood of $X^l_{t+1}$.If $X_{t}^l=Y_t$ and the edge $(X_t^l, X^l_{t+1})$ is in $V^*$, then couple $X^l$ and $Y$ so that $X^l_{t+1}=Y_{t+1}$. Namely, $X^l$ and $Y$ coincide until $X^l$ exits the graph $V^*$.   
\item The next unwanted event we wish to exclude is that the walk exits $V^*$ too quickly.
Set
\begin{align*}
\tau:=\begin{cases}
&2^{\gamma k+1}\hspace{56mm}\ \text{if}\ X_{i+t}^l= Y_t\ \ \text{for}\ \ \forall t\in[1,2^{\gamma k+1}],\\
&\min\{1\leq t\leq2^{\gamma k+1}:\ X_{i+t}^l\neq Y_t\}\hspace{15mm}\text{otherwise}.
\end{cases}
\end{align*}
By definition, we have $\tau= 2^{\gamma k+1}$ if $X^l_t\in V^*$ for all $t\in[i,i+2^{\gamma k+1}-1]$,
and the walk's transitions between $B(X_i^l,2^{\delta k})$ $B(x,2^{\delta k})$ are only through $(X_i^l,x)$.
It is shown in \cite[Claim 5.6]{crawford2013simple} that the probability that 
there is some special phase for which $\tau\neq 2^{\gamma k+1}$ is less than $2^{-\eta_2 k}$.
This follows from the bounds on the probabilities of $\mathscr{E}(\varepsilon,\delta,k)$ and $\mathscr{G}(\varepsilon,\gamma,\delta,k)$.
If this occurs, we call it a {\it type five error}.
\item Let
$$\mathcal{K}:=\{\tau=2^{\gamma k+1}\}\cap\{\tau^*<2^{\gamma k+1}\ \ \text{and}\ \ Y_t\notin\{x\}\cup\{X_i^l\}
\ \text{for}\ \forall t\in[\tau^*,2^{\gamma k+1}]\}.$$
This event basically ensures that the walk stays in $V^*$ up to time $i+2^{\gamma k+1}-1$, but the walk escapes from $\{X_i^l\}\cup\{x\}$ by time $i+2^{\gamma k+1}-1$.
It is shown in \cite[Claim 5.6]{crawford2013simple} that the probability that 
there is some special phase for which $\mathcal{K}$ does not hold is less than $2^{-\eta_2 k}$.
This is a consequence of the bounds on the probabilities of $\mathscr{F}(\varepsilon,\gamma,k)$ and $\mathscr{G}(\varepsilon,\gamma,\delta,k)$.
If this occurs, we call it a {\it type six error}.
\end{itemize}
By the construction presented above, we have the following:
\begin{claim}\cite[Lemma 5.5]{crawford2013simple}
On the event $\mathcal{K}$, it holds that
\begin{align*}
R(X_i^l)\leq R(x)\Rightarrow\ X^l_{i+2^{\gamma k+1}-1}\in B(X_i^l,2^{\delta k})\ \ \ \text{and}\ \ \ 
R(X_i^l)> R(x)\Rightarrow\ X^l_{i+2^{\gamma k+1}-1}\in B(x,2^{\delta k}).
\end{align*}
\end{claim}
Finally, we define the {\it good event}\ $\mathcal{G}$ by
$$\mathcal{G}:=\bigcap_{f=1}^6\{\text{there is no type}\ f\ \text{error in the entire coupling procedure}\}.$$
Then by Lemma 5.7 in \cite{crawford2013simple}, we have that
\begin{align}\label{ineq:good event}
\mathbb{P}_{\mu}(\mathcal{G}^c)\leq 2^{-\eta_2 k}.
\end{align}
This inequality is of fundamental importance in the subsequent sections.\\

\section{Proof of Theorem \ref{thm:main}}
The most of arguments in the proof of \cite{crawford2013simple} can be applied for the problem discussed in the article after minor changes. 
However, there are two issues which need significant modifications.
The one is the treatment of short jumps, which we already discussed in Proposition \ref{prop:shortsteps}.
The other is the treatment of long jumps. Namely, in our setting a walk encounters long edges more often than in the setting of \cite{crawford2013simple} since we changed the definition of them. 
In their definition, long edges are ones of length greater than $\lambda 2^{k/\alpha}$, where $\lambda>0$ is a small positive constant and $2^k$ is the number of steps of the walk under consideration. This implies that for fixed $\lambda>0$, the number of long edges encountered by the walk up to time $2^k$ is tight in $k$.\par
However, in this article we defined long edges to be those of length greater than $2^{(1/\alpha-\varepsilon)k}$ for a small constant $\varepsilon>0$, and proved in Proposition \ref{prop:shortsteps} that 
the contribution of jumps along short edges in our sense is negligible after the scaling $2^{-k/\alpha}$. 
Therefore, in our setting the number of the walk's encounters with long edges is no longer tight. This fact requires a nontrivial improvement of the arguments in Subsection 6.2 of \cite{crawford2013simple}.
To do so, we will need the following estimate.
Recall that
\begin{align}\label{def:phi}
\phi^l_i:=\sum_{j=1}^i {\bf 1}\left\{X^l_j\notin\{X^l_0,X^l_1,...,X^l_{j-1}\}\ \text{and}\ X^l_j\neq X^{l'}_{j'}\ \text{for any}\ l'< l, j'\in[2^k] \right\}.
\end{align}
Let a deterministic constant $\hat{C}\in(0,1)$ be the almost-sure limit of $\frac{\phi^l_n}{n}$. See Subsection 4.1 of \cite{crawford2013simple} for details.
\begin{proposition}\label{prop:phi}{\rm(Sub-linear fluctuations of counting processes)}
There exists an $\eta_3=\eta_3(\varepsilon,\delta,\gamma,d,s)>0$ such that 
$$2^{-(1-\eta_3)k}\max_{l\in[k^3]}\max_{1\leq i\leq 2^k}|\phi^l_i-i\hat{C}|$$
 is a tight sequence with respect to $\annealedP_{\mu}$.
\end{proposition}
We will prove Proposition \ref{prop:phi} in Section 7.

\subsection{The first approximation $(\hat{X}^l_i)$}
Let us discuss approximations, which are basically same as ones discussed in Section 6 of \cite{crawford2013simple}. 
However, there is one notable difference: in what follows, we will not use types of vertices introduced in \cite{crawford2013simple}.
Let $(X_i^l)$ be the $l$-th walk. 
We define the first approximation $(\hat{X}_i^l)$ below. For $i\in[2^k]$ and $l\in[k^3]$, define
\begin{align*}
\sigma_i^{l}:={\bf 1}\left\{R_i^{l}\left(\dfrac{(1-\tilde{p}^l_{X_{i-1}})\tilde{d}^{\omega}(X_{i-1})}{1+(1-\tilde{p}^l_{X_{i-1}})\tilde{d}^{\omega}(X_{i-1})}\right)
>
\tilde{R}_i^{l}\left(\dfrac{(1-\mathfrak{r}_i^{l})\mathfrak{d}_i^{l}}{1+(1-\mathfrak{r}_i^{l})\mathfrak{d}_i^{l}}\right)
\right\}.
\end{align*}
This indicator describes the parity of the number of crossings of a long edge. Namely, the indicator is $1$ when it is odd and is $0$ when it is even. 
Define
\begin{align*}
Z_{i}^{l}:=\sigma_i^{l}\sum_{|x|>2^{(1/\alpha-\varepsilon)k}}xw_i^{l}(x)\ \ \ \text{and}
\ \ \ \tilde{Z}_{i}^{l}:=\sigma_i^{l}\sum_{2^{(1/\alpha-\varepsilon)k}<|x|\leq 2^{(1/\alpha-\varepsilon_1)k}}xw_i^{l}(x),
\end{align*}
where $\varepsilon_1\in(0,\varepsilon)$ is another small positive constant we will choose later. See \eqref{def:eps1}.\par
We now define the first approximation $(\hat{X}^l_i)$ as follows:
\begin{align*}
\hat{X}^l_i&:=\sum_{i'=1}^{\phi_i^{l}}Z_{i'}^{l}=\sum_{i'=1}^{\phi_i^{l}} \sigma_i^{l}\sum_{|x|>2^{(1/\alpha-\varepsilon)k}}xw_i^{l}(x),
\end{align*}
where $\phi^l_i$, already defined in \eqref{def:phi}, is the number of undiscovered vertices the $l$-th walk encountered up to time $i$. 
In the next subsection, will use the following processes too:
\begin{align}\label{def:tilde-x}
\tilde{X}^l_i&:=\sum_{i'=1}^{\phi_i^{l}}\tilde{Z}_{i'}^{l}=\sum_{i'=1}^{\phi_i^{l}} \sigma_i^{l}\sum_{|x|>2^{(1/\alpha-\varepsilon_1)k}}xw_i^{l}(x),\nonumber\\
\tilde{\tilde{X}}^l_i&:=\sum_{i'=1}^{\phi_i^{l}} \sigma_i^{l}\sum_{2^{(1/\alpha-\varepsilon)k}<|x|\leq2^{(1/\alpha-\varepsilon_1)k}}xw_i^{l}(x),
\end{align}
so that $\hat{X}^l_i=\tilde{X}^l_i+\tilde{\tilde{X}}^l_i.$
We now introduce the rescaled processes as follows: for $t\in[0,1]$
\begin{align*}
X^{l}(t):=2^{-k/\alpha}X^l_{\lfloor t2^k\rfloor},\ \hat{X}^{l}(t):=2^{-k/\alpha}\hat{X}^l_{\lfloor t2^k\rfloor},\ \tilde{X}^{l}(t):=2^{-k/\alpha}\tilde{X}^l_{\lfloor t2^k\rfloor},\ \tilde{\tilde{X}}^{l}(t):=2^{-k/\alpha}\tilde{\tilde{X}}^l_{\lfloor t2^k\rfloor}.
\end{align*}
We will show that $\hat{X}^{l}$ approximates to $X^{l}$ in the following sense, which is a result corresponding to \cite[Corollary 6.3]{crawford2013simple}.
\begin{proposition}\label{prop:cs-6.3}
For any $1\leq q<\infty$ and $\xi>0$, there exists a constant $C>0$ such that for any $k\geq1$
\begin{align*}
\annealedP_{\mu}\left(\#\left\{l\in[k^3]:\|\hat{X}^{l}(t)-X^{l}(t)\|_{L^q[0,1]}\geq\xi\right\}>\xi k^3\right)\leq =O(k^{-3})
\end{align*}
\end{proposition}
\begin{proof}
The proof is basically same as the arguments in Subsection 6.1 of \cite{crawford2013simple}. We only mention several thing that need to be modified. In \cite{crawford2013simple}, they used the formula (16) in page 468 to control the sum of short jumps for the case $\alpha\in(0,1)$. In our setting, we can easily deduce the same estimate by using Proposition \ref{prop:shortsteps}.
\par
Secondly, in \cite{crawford2013simple}, they introduced the quantity
$$Z_{\rm max}^l:=\sum_{i=1}^{2^k}\sum_{(j,m)\in\{(0,0)\}\cup[J]^2}\sum_{|x|>\lambda 2^{k/\alpha}}|x|w^{l,j,m}_{i}(x),$$
 and used it to bound $\|\hat{X}^{l}(t)-X^{l}(t)\|_{L^q[0,1]}$
 in the proofs of Lemma 6.2 and Corollary 6.3. 
However, this quantity diverges in our setting. A modification required is to not use types of vertices and to replace $Z^l_{\rm max}$ with
\begin{align}\label{def:zmax}
\hat{Z}^l_{\rm max}:=\max_{1\leq i\leq 2^k}\sum_{|x|>2^{(1/\alpha-\varepsilon)k}}|x|w^{l}_{i}(x).
\end{align}
We can bound $\hat{Z}^l_{\rm max}$ using the same argument as in Lemma 6.4 in \cite{crawford2013simple}. This concludes the proof.
\end{proof}

\subsection{The second approximation $(\mathfrak{X}^l_i)$}
In this subsection, we will discuss the second approximation to replace $(\hat{X}^l_i)$ with a new process $(\hat{\mathfrak{X}}^l_i)$, which behaves in a simpler manner. This step consists in replacing $\phi^{l}_i$ with a deterministic function $i\hat{C}$, where 
$\hat{C}$ is the almost-sure limit of $\phi^l_n/n.$
Namely, we set
\begin{align*}
\hat{\mathfrak{X}}^l_i:=\sum_{i'=1}^{\lfloor i\hat{C}\rfloor}Z^{l}_{i'}=\sum_{i'=1}^{\lfloor i\hat{C}\rfloor}\sigma_{i'}^{l}\sum_{|x|>2^{(1/\alpha-\varepsilon)k}}xw_{i'}^{l}(x),
\end{align*}
and $\hat{\mathfrak{X}}^l(t):=2^{-k/\alpha}\hat{\mathfrak{X}}^l_{\lfloor t2^k\rfloor}$ for $t\in[0,1]$.\par
In \cite{crawford2013simple}, the authors used an ergodic argument to complete this step, but we have to employ more quantitative arguments since we have to deal with a larger number of encounters with long edges. We wish to show that $(\hat{\mathfrak{X}}^l_i)$ actually approximates to $(\hat{X}^{l}_i)$. 
To do so, 
 it suffices to show the following claim.

\begin{proposition}\label{prop:approx}
For any $1\leq q<\infty$, $\xi>0$ and for sufficiently large $k\geq1$
\begin{align*}
\annealedP_{\mu}\left(\#\left\{l\in[k^3]:\|\hat{\mathfrak{X}}^{l}(t)-\hat{X}^{l}(t)\|_{L^q[0,1]}\geq\xi\right\}>\xi k^3\right)= O(k^{-3}).
\end{align*}

\end{proposition}

Before proving Proposition \ref{prop:approx}, we show an auxiliary lemma. What we do is to introduce
an intermediate scale for the length of edges and decompose $(\hat{\mathfrak{X}}^l_i)$ further by using it.
Namely, we choose a small number $\varepsilon_1>0$ satisfying
\begin{align}\label{def:eps1}
\varepsilon_1<\varepsilon\ \ \text{and}\ \ \alpha\varepsilon_1<\eta_3,
\end{align}
where $\alpha:=s-d$ and $\eta_3=\eta_3(\varepsilon)$ is the constant appearing in Proposition \ref{prop:phi}.
The meaning of the condition \eqref{def:eps1} will become clear in the proof of Proposition \ref{prop:approx}.
We then decompose $(\hat{\mathfrak{X}}^l_i)$ as follows:
\begin{align*}
\hat{\mathfrak{X}}^l_i=\hat{\mathfrak{M}}^l_i+\hat{\mathfrak{N}}^l_i,
\end{align*}
where
\begin{align}\label{def:decomp}
\hat{\mathfrak{M}}^l_i&:=\sum_{i'=1}^{\lfloor i\hat{C}\rfloor}\sigma_{i'}^{l}\sum_{|x|>2^{(1/\alpha-\varepsilon_1)k}}xw_{i'}^{l}(x),\\
\hat{\mathfrak{N}}^l_i&:=\sum_{i'=1}^{\lfloor i\hat{C}\rfloor}\sigma_{i'}^{l}\sum_{2^{(1/\alpha-\varepsilon)k}<|x|\leq2^{(1/\alpha-\varepsilon_1)k}}xw_{i'}^{l}(x).
\end{align}
Set $\hat{\mathfrak{M}}^l(t):=2^{-k/\alpha}\hat{\mathfrak{M}}^l_{\lfloor t2^k\rfloor}, \hat{\mathfrak{N}}^l(t):=2^{-k/\alpha}\hat{\mathfrak{N}}^l_{\lfloor t2^k\rfloor}$ for $t\in[0,1]$. We first show that $\hat{\mathfrak{N}}^l(t)$, which is the rescaled sum of jumps of intermediate length, is negligible in the end.

\begin{lemma}\label{lem:aux}
For any $1\leq q<\infty$, $\xi>0$ and for sufficiently large $k\geq1$, we have that
\begin{align*}
\annealedP_{\mu}\left(\#\left\{l\in[k^3]:\|\tilde{\tilde{X}}^{l}(t)\|_{L^q[0,1]}\geq\xi\right\}>\xi k^3\right)=O(k^{-3})
\end{align*}
and
\begin{align*}
\annealedP_{\mu}\left(\#\left\{l\in[k^3]:\|\hat{\mathfrak{N}}^{l}(t)\|_{L^q[0,1]}\geq\xi\right\}>\xi k^3\right)= O(k^{-3}).
\end{align*}
See \eqref{def:tilde-x} and the formula below for the definition of $\tilde{\tilde{X}}$.
\end{lemma}
\begin{proof}
 For some $C>0$ we have that 
\begin{align*}
\mathbb{E}_{\mu}\left[\left|\sum_{2^{(1/\alpha-\varepsilon)k}<|x|\leq2^{(1/\alpha-\varepsilon_1)k}}xw^1_1(x)\right|^2\right]&\leq C\sum_{2^{(1/\alpha-\varepsilon)k}<l\leq2^{(1/\alpha-\varepsilon_1)k}} l^2\cdot l^{-s}\cdot l^{d-1}\\
&\leq C 2^{(1/\alpha-\varepsilon_1)(2-\alpha)k}.
\end{align*}
Furthermore on the good event $\mathcal{G}$, $\hat{\mathfrak{N}}^l_i$ is the sum of {\it i.i.d.} random variables
with finite second moment.
Therefore, we have that
\begin{align*}
\mathbb{E}_{\mu}\left[(2^{-k/\alpha}\hat{\mathfrak{N}}^l_{2^k})^2\right]&\leq C 2^{-2k/\alpha}\cdot 2^{(1/\alpha-\varepsilon_1)(2-\alpha)k}\cdot 2^k=C2^{-(2-\alpha)\varepsilon_1k}.
\end{align*}
 This estimate together with the martingale maximum inequality
yields the desired estimates.
\end{proof}

We finally prove Proposition \ref{prop:approx}.

\begin{proof}[Proof of Proposition \ref{prop:approx}]
By Lemma \ref{lem:aux}, the claim follows from the following estimate:
\begin{align*}
\annealedP_{\mu}\left(\#\left\{l\in[k^3]:\|\hat{\mathfrak{M}}^{l}(t)-\tilde{X}^{l}(t)\|_{L^q[0,1]}\geq\xi\right\}>\xi k^3\right)= O(k^{-3})
\end{align*}
for any $q\in[1,\infty)$ and $\xi>0$.
We mainly consider the case $l=1$ in what follows but the claim for general $l\in[k^3]$ follows from the same argument.
Write $ \langle a,b\rangle:=[a\wedge b,a\vee b]$, then we have that
\begin{align*}
\left\|\hat{\mathfrak{M}}^{l}(t)-\tilde{X}^{l}(t)\right\|_{L^q[0,1]} 
&\leq 2^{-k/\alpha}\hat{Z}^1_{\rm max}
\cdot\left\|\#\left\{ j\in\langle \phi_{\lfloor t2^k\rfloor}^{1}, \lfloor t2^k\rfloor \hat{C} \rangle :\tilde{Z}_{j}^{1}\neq0\right\}\right\|_{L^q[0,1]}\\
&\leq 2^{-k/\alpha}\hat{Z}^1_{\rm max}\cdot \left\{2^{-k}\sum_{i=1}^{2^k} \left(\#\left\{ j\in\langle \phi_i^{1}, i \hat{C} \rangle :\tilde{Z}_{j}^{1}\neq0\right\}\right)^q \right\}^{1/q}.
\end{align*}
Define 
$$\mathcal{V}^{l}:=2^{-k}\sum_{i=1}^{2^k} \left(\#\left\{ j\in\langle \phi_i^{l}, i \hat{C} \rangle :\tilde{Z}_{j}^{l}\neq0\right\}\right)^q.$$
Then by Lemma 6.4 in \cite{crawford2013simple}, in order to get the conclusion it suffices to prove that
\begin{align}\label{ineq:cher}
\annealedP_{\mu}\left(\#\left\{l\in[k^3]\ :\ \mathcal{V}^l >\xi\right\}>\xi k^3\right)=O(k^{-3}).
\end{align}
 To this end we will prove $\lim_{k\to\infty}\annealedP_{\mu}(\mathcal{V}^1 >\xi)=0$.
 This implies $\lim_{k\to\infty}P^{\omega}(\mathcal{V}^1 >\xi)=0$ for $\mu$-almost every $\omega$, 
 which indeed leads to \eqref{ineq:cher} by the Chernoff bound.\par
Let $(v_m)_{m\geq 1}$ be  i.i.d. Bernoulli random variables with 
$$\annealedP_{\mu}(v_1=1)=1-\annealedP_{\mu}(v_1=0)=\annealedP_{\mu}\left(\sum_{|x|>2^{(1/\alpha-\varepsilon_1)k}}xw^1_1(x)\neq0\right)=O(2^{(\alpha\varepsilon_1-1)k}).$$
Let $\upsilon>0$ be an arbitrary small number. By Proposition \ref{prop:phi}, there exists $K>0$ such that 
$$\annealedP_{\mu}\left(\sup_{1\leq i\leq 2^k}|\phi^1_i-i\hat{C}|>K2^{(1-\eta_3)k}\right)<\upsilon.$$
On the event $\{\sup_{1\leq i\leq 2^k}|\phi^1_i-i\hat{C}|\leq K2^{(1-\eta_3)k}\}$, it hold that
\begin{align*}
\mathcal{V}^1&\leq 2^{-k}\sum_{i=1}^{2^k}\left(\#\left\{ j\in[ i\hat{C}-K2^{(1-\eta_3)k}, i \hat{C}+K2^{(1-\eta_3)k} ] :\tilde{Z}_{j}^{l}\neq0\right\}\right)^q.
\end{align*}
Observe that $\left(\#\left\{ j\in[ i\hat{C}-K2^{(1-\eta_3)k}, i \hat{C}+K2^{(1-\eta_3)k} ] :\tilde{Z}_{j}^{l}\neq0\right\}\right)^q$ is stochastically dominated by 
$$\left(\sum_{m=1}^{2K\cdot 2^{(1-\eta_3)k}}v_m\right)^q.$$
Therefore,
\begin{align*}
&\ \ \ \ \annealedE_{\mu}\left[2^{-k}\sum_{i=1}^{2^k}\left(\#\left\{ j\in[ i\hat{C}-K2^{(1-\eta_3)k}, i \hat{C}+K2^{(1-\eta_3)k} ] :\tilde{Z}_{j}^{l}\neq0\right\}\right)^q\right]\\
&\leq\annealedE_{\mu}\left[\left(\sum_{m=1}^{2K\cdot 2^{(1-\eta_3)k}}v_m\right)^q\right]\\
&= O\left(\annealedE_{\mu}\left[\sum_{m=1}^{2K\cdot 2^{(1-\eta_3)k}}v_m\right]^q\right)\\
&\leq O\left(\left(2^{(1-\eta_3)k}2^{(\alpha\varepsilon_1-1)k}\right)^q\right)\\
&\leq O\left(2^{q(\alpha\varepsilon_1-\eta_3)k}\right).
\end{align*}
By \eqref{def:eps1}, this estimate implies that 
$\limsup_{k\to\infty}\annealedP_{\mu}(\mathcal{V}^1>\xi)<\upsilon$. Since $\upsilon$ is arbitrary, we get the conclusion.
\end{proof}

\subsection{On the structure of $\hat{\mathfrak{X}}^l_i$}
In this subsection, we will investigate the structure of $\hat{\mathfrak{X}}^l_i$. Recall that
\begin{align*}
\hat{\mathfrak{X}}^l_i&:=\sum_{i'=1}^{\lfloor i\hat{C}\rfloor}Z_{i'}^l
=\sum_{i'=1}^{\lfloor i\hat{C}\rfloor}\sigma_{i'}^{l}\sum_{|x|>2^{(1/\alpha-\varepsilon)k}}xw_{i'}^{l}(x)\\
&=\sum_{i'=1}^{\lfloor i\hat{C}\rfloor}
{\bf 1}\left\{R_i^{l}\left(\dfrac{(1-\tilde{p}^l_{X_{i-1}})\tilde{d}^{\omega}(X_{i-1})}{1+(1-\tilde{p}^l_{X_{i-1}})\tilde{d}^{\omega}(X_{i-1})}\right)
>
\tilde{R}_i^{l}\left(\dfrac{(1-\mathfrak{r}_i^{l})\mathfrak{d}_i^{l}}{1+(1-\mathfrak{r}_i^{l})\mathfrak{d}_i^{l}}\right)
\right\}
\sum_{|x|>2^{(1/\alpha-\varepsilon)k}}xw_{i'}^{l}(x).
\end{align*}
We now observe the following.
\begin{lemma}\label{lem:indep}
\begin{enumerate}
\item[(i)] The sequence $(\sigma_j^l)_{1\leq j\leq 2^k}$ is independent of $(\sum_{|x|>2^{(1/\alpha-\varepsilon)k}}xw_{j}^{l}(x))_{1\leq j\leq 2^k}$ under $\annealedP_{\mu}$.
\item[(ii)] On the good event $\mathcal{G}$, under $\annealedP_{\mu}$, $(Z_j^l)_{1\leq j\leq 2^k}$ is an independent sequence.
\end{enumerate}
\end{lemma}
\begin{proof}
\begin{enumerate}
\item[(i)] This immediately follows from the definitions of local return probabilities $(\tilde{p}_x)_{x\in\mathbb{Z}^d}$ and local degrees $(\tilde{d}^{\omega}(x))_{x\in\mathbb{Z}^d}$.
\item[(ii)] The second claim follows from the following observations:
\begin{itemize}
\item $Z^l_i=0$ unless $(X^l_j)$ encounters a new long edge at time $i$.
\item On $\mathcal{G}$, the distance between two distinct long edges encountered by $(X^l_j)$ is at least 
$2\cdot2^{\delta k}$. 
\end{itemize}
The first claim holds by definition.
The second claim follows from the construction of the coupling in \cite{crawford2013simple} 
See the definition of the type four error. 
\end{enumerate}
\end{proof}

Let $(\tilde{w}_i^l(x))_{i\geq1}$ be i.i.d. Bernoulli random variables whose distribution is identical to
that of 
$(w_i^l(x))_{i\geq1}$. Furthermore, assume that $(\tilde{w}_i^l(x))_{i\geq1}$ is independent of $(w_i^l(x))_{i\geq1})$ and $(\sigma_{i}^l)_{i\geq1}$. Define
\begin{align*}
\tilde{\mathfrak{X}}_i^l&:=\hat{\mathfrak{X}}^l_i+\sum_{i'=1}^{\lfloor i\hat{C}\rfloor}\sigma_{i'}^{l}\sum_{|x|\leq 2^{(1/\alpha-\varepsilon)k}}x\tilde{w}_{i'}^{l}(x)\\
&=\sum_{i'=1}^{\lfloor i\hat{C}\rfloor}\sigma_{i'}^{l}\sum_{x\in\mathbb{Z}^d}x\left(\tilde{w}_{i'}^{l}(x){\bf 1}_{\{|x|\leq 2^{(1/\alpha-\varepsilon)k}\}}+w_{i'}^l(x){\bf 1}_{\{|x|>2^{(1/\alpha-\varepsilon)k}\}}  \right)
\end{align*}
and $\tilde{\mathfrak{X}}^l(t):=2^{-k/\alpha}\tilde{\mathfrak{X}}^l_{\lfloor t2^k \rfloor}.$
By the same argument as in the proof of Lemma 6.7 in \cite{crawford2013simple}, we get the following.

\begin{proposition}\label{prop:tilde} 
For any $1\leq q<\infty$, $\xi>0$ and for sufficiently large $k\geq1$
\begin{align*}
\annealedP_{\mu}\left(\#\left\{l\in[k^3]:\|\hat{\mathfrak{X}}^{l}(t)-\tilde{\mathfrak{X}}^{l}(t)\|_{L^q[0,1]}\geq\xi\right\}>\xi k^3\right)\leq O(k^{-3}).
\end{align*}
\end{proposition}

\subsection{The final step: approximating $\tilde{\mathfrak{X}}^1(t)$ by an $\alpha$-stable process}
This step can be completed by Lemma \ref{lem:indep} and Proposition \ref{prop:tilde}, therefore Theorem \ref{thm:main} finally follows from the arguments in Section 7 of \cite{crawford2013simple}.

\section{Proof of Proposition \ref{prop:phi}}

In this section, we will prove Proposition \ref{prop:phi}, which is a key technical estimate for the proof of Theorem \ref{thm:main}. 
In order to prove Proposition \ref{prop:phi}, we need to estimate how much time is required for the walk to find a new long edge. 
In what follows, we only consider the first random walk $X^1_n$.\par
Set $\tilde{m}_0=0$ and for each $j\in\mathbb{N}$, let $\tilde{m}_j$ be the first time the walk visits the $j$-th new long edge $(v_j,x_j)$ at the vertex $v_j\ (0=\tilde{m}_0< \tilde{m}_1<\tilde{m}_2<...<\tilde{m}_{j}<...)$. Namely, 
$X^1_{\tilde{m}_j}=v_j, |v_j-x_j|>2^{(1/\alpha-\varepsilon)k}$ and $\omega_{v_j,x_j}=1$.\par
We next choose a subsequence $(m_j)$ of $(\tilde{m}_j)$ in the following manner: 
\begin{align*}
m_0=0\ \ \text{and}\ \ m_1&:=\tilde{m}_{l(1)}, \ \text{where}\ \ l(1):=\inf\{l\in\mathbb{N}:\ X^1_{\tilde{m}_l+2^{\gamma k+1}}\in B(x_l,2^{\delta k})\},\\
\intertext{and for $j\geq2$,}
m_j&:=\tilde{m}_{l(j)}, \ \text{where}\ \ l(j):=\inf\{l> l(j-1):\ X^1_{\tilde{m}_l+2^{\gamma k+1}}\in B(x_l,2^{\delta k})\}.
\end{align*}
Namely, we extract members of $(\tilde{m}_j)$ after $2^{\gamma k+1}$ time of which the walk stays within $B(x_j,2^{\delta k})$.
This extraction is needed to ensure that the walk crosses the new long edge odd times so that it discovers a fresh (independent) environment.
Note that on $\mathcal{G}$ these variables are finite a.s.
and $(\tilde{m}_j)$ and $(m_j+2^{\gamma k+1})$ are stopping times though $(m_j)$ are not.
 Let 
\begin{align}\label{def:beta}
\tilde{\beta}(k)=\tilde{\beta}(k,\varepsilon):=\max\{j: \tilde{m}_j\leq2^k\},\ \beta(k)=\beta(k,\varepsilon):=\max\{j: m_j\leq2^k\}.
\end{align}

 A basic idea is to divide the walk's path into blocks according to the increasing sequence $m_1<m_2<...<m_j<...<m_{\beta(k)}\leq 2^k$. In order to implement this idea, we first show the following estimate.
 \begin{lemma}\label{lem:tilde-m}
 Suppose that $d\geq1$ and $s\in(d,(d+2)\wedge 2d)$.
Then there exists an $\eta_4=\eta_4(\varepsilon)>0$ such that
\begin{align}\label{est:m_j}
\annealedE_{\mu}\left[\max_{0\leq j\leq\tilde{\beta}(k)}(\tilde{m}_{j+1}-\tilde{m}_j)^2\right]=O(2^{(2-\eta_4)k}).
\end{align}
\end{lemma}
 \begin{proof}
The proof is a verification of the following simple observations: while the walk is seeking for a new long edge, it moves only along short ones. Therefore, the displacement of the walk cannot be too large by
Proposition \ref{prop:shortsteps}. On the other hand, the probability that the walk is confined within a small ball (compared with its time-scale $2^{k/\alpha}$) is controlled by Proposition \ref{prop:hke}, too.\par
We now turn this observation into a rigorous argument.
We will prove that there exist small positive constants $\eta'_4=\eta'_4(\varepsilon), \kappa=\kappa(\varepsilon)$ such that
\begin{align}\label{est:max-m_j}
\annealedP_{\mu}\left(\max_{0\leq j\leq \tilde{\beta}(k)}(\tilde{m}_{j+1}-\tilde{m}_j)\geq 2^{(1-\kappa)k}\right)
\leq 2^{-\eta'_4 k}.
\end{align}
This estimate, H\"{o}lder's inequality and the obvious fact that $\max_{0\leq j\leq\beta(k)}(\tilde{m}_{j+1}-\tilde{m}_j)\leq 2^k$ imply the claim \eqref{est:m_j}.\par
Notice that the random variable $\tilde{\beta}(k)$ is stochastically dominated by $\sum_{i=1}^{2^k}v_i$,
where $(\tilde{v}_i)$ is a sequence i.i.d. Bernoulli random variables with 
$$\annealedP_{\mu}(v_1=1)=1-\annealedP_{\mu}(v_1=0)=\annealedP_{\mu}\left(\sum_{|x|>2^{(1/\alpha-\varepsilon)k}}xw^1_1(x)\neq0\right)=O(2^{(\alpha\varepsilon-1)k}).$$
Therefore, by using H\"{o}lder's inequality and computing exponential moments of $\sum_{i=1}^{2^k}v_i$,  it holds that
\begin{align}\label{ineq:sto-dom}
\annealedP_{\mu}\left( \tilde{\beta}(k)\geq 2^{2\alpha\varepsilon k}\right)\leq 2^{-\psi k}
\end{align}
for some $\psi>0$.
We next define for $1\leq n\leq m$
\begin{align*}
W_{k}(n,m):=\max_{n\leq l\leq m} \left| \frac{1}{2^{k/\alpha}} \sum_{j=1}^l (X_j^1 - X_{j-1}^1)\cdot{\bf 1}_{\left\{\left|X_j^1 - X_{j-1}^1\right| \leq 2^{k(1/\alpha - \hezka)}\right\}} \right|.
\end{align*}
Recall that $\rho:=\frac{(2-\alpha)\varepsilon}{2}$, see Proposition \ref{prop:shortsteps}.
For $\xi\in(0,\rho)$, we have that
\begin{align}\label{ineq:tedious}
&\ \ \ \ \annealedP_{\mu}\left(\max_{0\leq j\leq \tilde{\beta}(k)}(\tilde{m}_{j+1}-\tilde{m
}_j)\geq 2^{(1-\kappa)k}\right)\nonumber\\
&\leq\annealedP_{\mu}\left(\max_{0\leq j\leq \tilde{\beta}(k)}(\tilde{m}_{j+1}-\tilde{m}_j)\geq2^{(1-\kappa)k},
\ \max_{0\leq j\leq \tilde{\beta}(k)}W_k(\tilde{m}_j,\tilde{m}_{j+1})\leq 2^{-(\rho-\xi) k }\right)\nonumber\\
&\hspace{20mm}+\annealedP_{\mu}\left(\max_{0\leq j\leq \tilde{\beta}(k)}W_k(\tilde{m}_j,\tilde{m}_{j+1})\geq 2^{-(\rho-\xi) k }\right).
\end{align}
The first term can be bounded as follows: by Proposition \ref{prop:hke} and \eqref{ineq:sto-dom}, we get that
\begin{align}\label{ineq:term1}
&\annealedP_{\mu}\left(\max_{0\leq j\leq \tilde{\beta}(k)}(\tilde{m}_{j+1}-\tilde{m}_j)\geq2^{(1-\kappa)k},
\ \max_{0\leq j\leq \tilde{\beta}(k)}W_k(\tilde{m}_j,\tilde{m}_{j+1})\leq 2^{-(\rho-\xi) k }\right)\nonumber\\
\leq&2^{-\psi k}+\annealedP_{\mu}\Bigl(\max_{0\leq j\leq 2^{2\alpha\varepsilon k}}(\tilde{m}_{j+1}-\tilde{m}_j)\geq2^{(1-\kappa)k},\ \tilde{m}_{2^{2\alpha\varepsilon k}}\leq 2^k\ \text{and}\nonumber\\
&\ \ \ \ \ \ X^1_{\tilde{m}_j+2^{(1-\kappa)k-1}}\in B\left(X_{\tilde{m}_j}^1, 2^{k/\alpha}\cdot 2^{-(\rho-\xi) k }\right)\ \text{for}\ \forall j\ \text{s.t.}\ \tilde{m}_{j+1}-\tilde{m}_j\geq 2^{(1-\kappa)k}\Bigr)\nonumber\\
\leq& \left(2^{(1-\kappa)k-1}\right)^{-d/\alpha} \left(2^{k/\alpha}\cdot 2^{-(\rho-\xi) k }\right)^d\cdot
\sum_{j=0}^{2^{2\alpha\varepsilon k}}\mathbb{P}_{\mu}\left(m_{j+1}-m_j\geq 2^{(1-\kappa)k},\ m_{j+1}\leq 2^k\right)\nonumber\\
\leq& C2^{(\kappa/\alpha-\rho+\xi)dk}\cdot 2^{\kappa k}.
\end{align}
We choose $\kappa$ and $\xi$ small enough so that
$$\kappa+(\kappa/\alpha-\rho+\xi)d<0.$$
Note that in the last step we used the fact that 
$$\sum_{j=0}^{2^{2\alpha\varepsilon k}}\mathbb{P}_{\mu}\left(m_{j+1}-m_j\geq 2^{(1-\kappa)k},\ m_{j+1}\leq 2^k\right)\leq 2^{\kappa k}.$$
The second term can be bounded as follows: 
by triangle inequality, we have 
$$\max_{0\leq j\leq \beta(k)}W_k(\tilde{m}_j,\tilde{m}_{j+1})\leq 2 W_k(0,2^k).$$
Therefore by using \eqref{ineq:short-condition} and the Markov inequality, we obtain that
\begin{align}\label{ineq:term3}
\annealedP_{\mu}\left(\max_{0\leq j\leq \beta(k)}W_k(\tilde{m}_j,\tilde{m}_{j+1})\geq 2^{-(\rho-\xi) k }\right)= O(2^{-\xi k}). 
\end{align}
Combining \eqref{ineq:tedious}, \eqref{ineq:term1} and \eqref{ineq:term3}, we get the desired estimate.
\end{proof}

We next prove a similar estimate for $(m_j)$.
\begin{lemma}\label{lem:new edge}
Suppose that $d\geq1$ and $s\in(d,(d+2)\wedge 2d)$. Then we have that
\begin{align}\label{est:m_j}
\annealedE_{\mu}\left[\max_{0\leq j\leq\beta(k)}(m_{j+1}-m_j)^2\right]=O(2^{(2-\eta_4/2)k}).
\end{align}
\end{lemma}
\begin{proof}
By Lemma \ref{lem:tilde-m}, it suffices to show that
\begin{align}\label{est:max-m_j}
\annealedP_{\mu}\left(\max_{0\leq j\leq \beta(k)}(\tilde{m}_{j+1}-\tilde{m}_j)\leq 2^{(1-\kappa)k}\ \ \text{and}\ \ \max_{0\leq j\leq \beta(k)}(m_{j+1}-m_j)\geq 2^{(1-\kappa/2)k}\right)
\leq 2^{-\eta'' k}
\end{align}
for some $\kappa',\eta''>0$.
Recall that $\tilde{m}_j$ is the first time the walk visits the $j$-th new long edge $(v_j,x_j)$ at the vertex $v_j$. Therefore, we get that
\begin{align}\label{ineq:very-long}
&\annealedP_{\mu}\left(\max_{0\leq j\leq \beta(k)}(\tilde{m}_{j+1}-\tilde{m}_j)\leq 2^{(1-\kappa)k}\ \ \text{and}\ \ \max_{0\leq j\leq \beta(k)}(m_{j+1}-m_j)\geq 2^{(1-\kappa/2)k}\right)\nonumber\\
\leq &\annealedP_{\mu}\left(\exists j\leq\tilde{\beta}(k)\ \text{s.t.}\  X^1_{\tilde{m}_{j'}+2^{\gamma k+1}-1}\notin B(x_j,2^{\delta k})\ \text{for}\ \forall j'\in[j,j+2^{\kappa k/2}]\right),\nonumber\\
\intertext{by \eqref{ineq:good event} and \eqref{ineq:sto-dom} we obtain that}
\leq &2^{-\eta_2 k}+2^{-\psi k}\nonumber\\
&+\sum_{j=1}^{2^{2\alpha\varepsilon k}}\annealedP_{\mu}\left(\left\{ X^1_{\tilde{m}_{j'}+2^{\gamma k+1}-1}\notin B(x_j,2^{\delta k})\ \text{for}\ \forall j'\in[j,j+2^{\kappa k/2}]\right\}\cap\{m_{2^{2\alpha\varepsilon k}}\leq 2^k\}\cap\mathcal{G}\right).
\end{align}
Let $K>0$ be large enough so that
\begin{align}\label{def:K}
\mu(\deg(0)\leq K)\geq \frac{1}{2}\ \ \text{and}\ \ \mu(\tilde{p}(0)\geq 1/K)\geq \frac{1}{2}.
\end{align}
Recall that $\tilde{p}(0)$ is a local return probability defined in Section \ref{sec:coup}. The estimates \eqref{def:K} imply that for any $j$,
\begin{align*}
&\annealedP_{\mu}\left(\left\{ X^1_{\tilde{m}_{j'}+2^{\gamma k+1}-1}\in B(x_j,2^{\delta k})\right\}\cap\{m_{j+1}\leq 2^k\}\cap\mathcal{G}\right)\\
\geq &\annealedP_{\mu}\left(\left\{X^1_{m_j+1}=x_j,\  X^1_{m_j+2}\neq v_j\ \text{and}\ X^1_i\neq x_j\ \text{for}\ \forall i\geq m_{j}+2\right\}\cap\mathcal{G}\right)\\
\geq&\left(\frac{1}{2}\right)^4\cdot K^{-3}=\frac{1}{16K^3}.
\end{align*}
This estimate together with \eqref{ineq:very-long} implies the claim.
\end{proof}

\begin{proof}[Proof of Proposition \ref{prop:phi}]
What we will do to prove Proposition \ref{prop:phi} is to decompose paths of a RW into several "blocks" by using $(m_j)$. We then prove and utilize the fact that distinct blocks are uncorrelated since a traversal along a new long edge leads the walk to a part of the environment that has not been explored before.
We verify these claims in the following way:
\begin{enumerate}[(i)]
\item the first step is to approximate the  indicator function 
$${\bf 1}\left\{X^l_j\notin\{X^l_0,X^l_1,...,X^l_{j-1}\}\ \text{and}\ X^l_j\neq X^{l'}_{j'}\ \text{for any}\ l'< l, j'\in[2^k] \right\}$$
 by
$$ {\bf 1}\left\{X^l_j\notin\{X^l_0,X^l_1,...,X^l_{j-1}\}\ \right\}.$$
 This is justified by Lemma 9.10 in \cite{crawford2013simple} which claims that intersections of two independent SRWs have a sub-linear growth. 
 Namely, it claims that for any $l,l'\in[k^3]$ with $l\neq l'$, there exists $c,\tilde{\upsilon}>0$ such that 
 $$\annealedE_{\mu}[|\{X_0^l,X_1^l,...,X_{2^k}^l\}\cap\{X_0^{l'},X_1^{l'},...\}|]\leq c2^{(1-\tilde{\upsilon})k},$$
 where for a set $A$, we denote by $|A|$ the cardinality of $A$.
 (NB. In Lemma 9.10 in \cite{crawford2013simple} they only claimed the above estimate for $s\in(d,d+1)$ but the proof extends to $d\geq2, s\in[d+1,d+2)$ without any change.)\par
 Set
 \begin{align}\label{def:tilde-phi}
 \tilde{\phi}_i^l:=\sum^i_{j=1}{\bf 1}\left\{X^l_j\notin\{X^l_0,X^l_1,...,X^l_{j-1}\}\ \right\}.
 \end{align}
  \item We next replace $ {\bf 1}\left\{X^l_i\notin\{X^l_0,X^l_1,...,X^l_{i-1}\}\ \right\}$ with
 $${\bf 1}\left\{X^l_i\notin\{X^l_0,X^l_1,...,X^l_{i-1}\}\right\} \cap\left\{\omega_{x,X^l_i}=0\ \text{for}\ \forall x\in\mathbb{Z}^d\ \text{with}\ |x-X_i^l|>2^{\delta k}\right\}.$$
 It is easy to check that the difference of sums of these indicators also have a sub-linear growth. 
 \item 
 We decompose the walk's path into blocks according to the increasing sequence $m_1<m_2<...<m_j<...<m_{\beta(k)}\leq 2^k$.
 We wish to obtain a kind of regenerative structure by doing this. To do so, we further replace the sum
 $$\sum_{i=1}^{2^k}{\bf 1}\left\{X^l_i\notin\{X^l_0,X^l_1,...,X^l_{i-1}\}\right\} \cap\left\{\omega_{x,X^l_i}=0\ \text{for}\ \forall x\in\mathbb{Z}^d\ \text{with}\ |x-X_i^l|>2^{\varepsilon_1 k}\right\}$$
 with
 $\sum_{j=1}^{\beta(k)}\mathcal{N}^l_j$, where
  \begin{align}\label{def:cal-n}
  \mathcal{N}^l_1&:=\sum_{i=0}^{m_1}{\bf 1}\left\{X^l_i\notin
 \{X^l_{0},X^l_{1},...,X^l_{i-1}\}\right\} \cap\left\{\omega_{x,X^l_i}=0\ \text{for}\ \forall x\in\mathbb{Z}^d\ \text{with}\ |x-X_i^l|>2^{\varepsilon_1 k}\right\}\nonumber\\
  \intertext{and for $2\leq j\leq\beta(k)$}
 \mathcal{N}^l_j&:=\sum_{i=m_{j-1}+2^{\gamma k+1}}^{m_j}{\bf 1}\left\{X^l_i\notin
 \{X^l_{m_{j-1}},...,X^l_{i-1}\}\right\} \cap\left\{|X^l_i-X^l_{i'}|>2^{\delta k}
 \ \text{for}\ \forall i'\leq m_{j-1}\right\}\nonumber\\
 &\hspace{55mm}\cap\left\{\omega_{x,X^l_i}=0\ \text{for}\ \forall x\in\mathbb{Z}^d\ \text{with}\ |x-X_i^l|>2^{\delta k}\right\},
 \end{align}
 By Corollary \ref{cor:no-return}, this replacement only gives rise to a sub-linear difference.

\item We finally claim that $\{\mathcal{N}^{l}_j\}_{1\leq j\leq \beta(k)-1}$ are independent on $\mathcal{G}$, but this is really immediate from the definition of $\mathcal{N}^l_j$ and the construction of the coupling explained in \cite[Section 5]{crawford2013simple}.
\end{enumerate}
We now aim at showing that there exists $\eta_3'=\eta_3'(\varepsilon)>0$ such that the following estimate holds:
\begin{align}\label{ineq:goal1}
\annealedE_{\mu}\left[\max_{1\leq i\leq 2^k}|\tilde{\phi}^l_{i}-i\hat{C}|\right]=O(2^{(1-\eta_3') k}).
\end{align}
This estimate together with the Markov inequality and the union bound implies the conclusion. 
To this end, we first deduce \eqref{ineq:goal1} from the following inequality we will prove later:
there exists $\eta_3''=\eta_3''(\varepsilon)>0$ such that 
\begin{align}\label{ineq:goal2}
\annealedE_{\mu}\left[\left(\sum_{j=0}^{\beta(k)-1}\left(\mathcal{N}_j^l-\annealedE_{\mu}\left[\mathcal{N}_j^l\right]\right)\right)^2\right]=O(2^{(2-\eta_3'')k}).
\end{align}
In what follows, we will only consider the case $l=1$ but the same argument implies the claim for general $l\in[k^3]$.

 \begin{proof}[Proof of $\eqref{ineq:goal2}\Rightarrow\eqref{ineq:goal1}$]
Since $\{\mathcal{N}^1_j\}_{1\leq j\leq\beta(k)}$ is an independent sequence, 
\begin{align*}
M_{i}:=\begin{cases}
&\sum_{j=0}^{i-1}\left(\mathcal{N}_j^1-\annealedE_{\mu}\left[\mathcal{N}_j^1\right]\right)\ \ \ \text{when}\ 1\leq i\leq \beta(k)\\
& M_{\beta(k)}\ \hspace{31mm}\ \text{when}\ i>\beta(k).
\end{cases}
\end{align*}
is a centered martingale. By the martingale maximal inequality and \eqref{ineq:goal2}, we get that
\begin{align}\label{ineq:mart}
\annealedE_{\mu}\left[\max_{1\leq i\leq\beta(k)}M_i\right]=O(2^{(1-\eta_3''/2)k}).
\end{align}
By the triangle inequality
\begin{align*}
&\annealedE_{\mu}\left[\max_{1\leq i\leq 2^k}|\tilde{\phi}^l_{i}-i\hat{C}|\right]\\
\leq &\annealedE_{\mu}\left[\max_{1\leq i\leq\beta(k)}M_i\right]+2^{\gamma k+1}\annealedE_{\mu}[\beta(k)]+2\annealedE_{\mu}\left[\max_{0\leq i\leq \beta(k)}(m_{j+1}-m_j)\right]+\annealedE_{\mu}[2^k-\beta(k)].
\end{align*}
The first three terms can be controlled by \eqref{ineq:sto-dom}, Lemma \ref{lem:new edge} and \eqref{ineq:mart}. It is easy to bound the last term using the observation above \eqref{ineq:sto-dom}.
\end{proof}

We finally prove \eqref{ineq:goal2} to get the conclusion. 
Overriding the rule of our notation, we write $m_{\beta(k)+1}=2^k$. By \eqref{ineq:sto-dom} and independence, it holds that
\begin{align*}
&\annealedE_{\mu}\left[\left(\sum_{j=1}^{\beta(k)}\left(\mathcal{N}_j^l-\annealedE_{\mu}\left[\mathcal{N}_j^l\right]\right)\right)^2\right]\\
\leq&\annealedE_{\mu}\left[\left(\sum_{j=1}^{\beta(k)}\left(\mathcal{N}_j^l-\annealedE_{\mu}\left[\mathcal{N}_j^l\right]\right)\right)^2:\mathcal{G}\ \right]+2^{2k}\annealedP_{\mu}(\mathcal{G}^c)\\
\leq &4C\annealedE_{\mu}\left[\left(\sum_{j=1}^{2^{2\alpha\varepsilon k}}\left(\mathcal{N}_j^l-\annealedE_{\mu}\left[\mathcal{N}_j^l\right]\right)\right)^2:\mathcal{G}\ \right]\\
&\hspace{30mm}+4C\cdot 2^{2k}\annealedP_{\mu}(\beta(k)\geq 2^{(1+\xi)\alpha\varepsilon})+4\cdot 2^{2k}\annealedP_{\mu}(\mathcal{G}^c)\\
\leq &4C\sum_{j=1}^{2^{2\alpha\varepsilon k}}\annealedE_{\nu}\left[\left(\mathcal{N}_j^1-\annealedE_{\mu}\left[\mathcal{N}_j^l\right]\right)^2:\mathcal{G}\ \right]+4C\cdot 2^{2k}\annealedP_{\mu}(\beta(k)\geq2^{(1+\xi)\alpha\varepsilon})+2^{2k}\annealedP_{\mu}(\mathcal{G}^c)\\
\leq&16C\sum_{j=0}^{2^{2\alpha\varepsilon k}  }\annealedE_{\nu}\left[\left(m_{j+1}-m_j\right)^2\right]+4C\cdot 2^{2k}\annealedP_{\mu}(\beta(k)\geq2^{(1+\xi)\alpha\varepsilon})+2^{2k}\annealedP_{\mu}(\mathcal{G}^c).
\end{align*}

By \eqref{ineq:good event}, \eqref{ineq:sto-dom} and Lemma \ref{lem:new edge}, we get the conclusion.
\end{proof}

\bibliography{LRP}
\bibliographystyle{plain}

\end{document}

%% file: local_macros.tex
\newcommand{\abdry}[2]{\partial_{\kern-0.3ex\scalebox{.6}{$#1$}}\kern-0.1ex{#2}}
\newcommand{\starsim}{%
  \mathrel{\vbox{\offinterlineskip\ialign{%
    \hfil##\hfil\cr
    \mbox{\scriptsize$*$}\cr
    \noalign{\kern0ex}
    $\sim$\cr
}}}}

